%% file: autCI.tex
\documentclass[final,oneside,notitlepage,10pt]{article}

\input{styles}

\input{kobib}
\input{makros}

\input{hyphenation_en}
\input{theorems}

\title{Categorical symmetries of T-duality}

\author{Konrad Waldorf}
\email{konrad.waldorf@uni-greifswald.de}

\adress{Universität Greifswald\\
Institut für Mathematik und Informatik\\
Walther-Rathenau-Str. 47\\
D-17487 Greifswald}

\newcommand{\TD}[1][]{\mathbb{TD}\ifx!#1!\else_{#1}\fi}
\newcommand{\AU}[1][]{\mathbb{A}\ifx!#1!\else_{#1}\fi}
\newcommand{\AUpos}[1][]{\mathbb{A}\ifx!#1!\else_{#1}\fi}
\newcommand{\TDpol}[1][]{\mathbb{TD}^{pol}\ifx!#1!\else_{#1}\fi}
\newcommand{\TB}[1][]{\mathbb{TB}\ifx!#1!\else_{#1}\fi}
\newcommand{\TBF}[1][]{\mathbb{TB}\ifx!#1!\else_{#1}\fi^{F1}}
\newcommand{\TBfib}[1][]{\mathbb{TB}\ifx!#1!\else_{#1}\fi^{fib}}
\newcommand{\TBFF}[1][]{\mathbb{TB}\ifx!#1!\else_{#1}\fi^{F2}}
\newcommand{\TBFFR}[1][]{\mathbb{TB}\ifx!#1!\else_{#1}\fi^{F2\text{-}\R}}
\newcommand{\TF}[1][]{\mathbb{TF}^{\pm}\ifx!#1!\else_{#1}\fi}
\newcommand{\TFpos}[1][]{\mathbb{TF}\ifx!#1!\else_{#1}\fi}
\newcommand{\TFpol}[1][]{\mathbb{TF}^{geo}\ifx!#1!\else_{#1}\fi}
\newcommand{\TFgeo}[1][]{\mathbb{TD}^{\frac{1}{2}\text{-}geo}\ifx!#1!\else_{#1}\fi}
\def\CI{\mathcal{C}r\mathcal{I}n}

\def\tcor{\mathrm{T}\text{-}\mathcal{C}orr}

\def\bra#1#2{\langle#1|#2|}
\def\braket#1#2#3{\bra{#1}{#2}#3\rangle}

\def\lbraket#1#2#3{\bra{#1}{#2}#3\rangle_{low}}

\def\B#1{\mathcal{B}(#1)}
\def\X#1#2{\mathcal{X}(#1,#2)}

\def\CMI{\mathcal{C}r\mathcal{M}od^{str}}
\def\CM{\mathcal{C}r\mathcal{M}od}
\def\sL2Gs{2\mathcal{G}r\!p^{str}}
\def\L2Gs{2\mathcal{G}r\!p}
\def\ssL2gh{\mathcal{H}om^{se-str}}

\msc{Primary 57T10, Secondary 20L05, 81T35}
\dedication{}

\denseversion

\usepackage{amstext}
\usepackage{amsmath}

\begin{document}

 \maketitle 

\begin{abstract}
Topological T-duality correspondences are higher categorical objects that can be classified by a strict Lie 2-group. In this article we compute the categorical automorphism group of this 2-group; hence, the categorical symmetries of topological T-duality. We prove that the categorical automorphism group is a non-central categorical extension of the integral split pseudo-orthogonal group. We show  that it splits over several  subgroups, and that its k-invariant is 2-torsion.    
\showmsc
\end{abstract}

\setcounter{tocdepth}{2}

\mytableofcontents

\setsecnumdepth{1}

\section{Introduction}

\label{sec:intro}

\setsecnumdepth{0}

Topological T-duality is a mathematical toy model for certain dualities arising in physics, most prominently in string theory compactifications \cite{Buscher1987,Giveon1994}, but also in topological insulators \cite{Mathai2015}.
It is centered around so-called \emph{T-duality correspondences}, diagrams
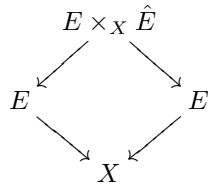
\begin{equation*}
\alxydim{@C=0.6em@R=1.5em}{& E \times_X \hat E \ar[dl] \ar[dr] \\ E \ar[dr] && E \ar[dl] \\ & X}
\end{equation*}
where $E$ and $\hat E$ are principal bundles for a torus group $\T^{n}=\ueins \times ... \times \ueins$, equipped with integral cohomology classes $\xi\in \h^3(E,\Z)$ and $\hat\xi\in \h^3(\hat E,\Z)$, subject to a certain condition formulated on the the correspondence space $E \times_X \hat E$. Dashing forward, physicists are then  interested in so-called \emph{T-folds}, structures that locally look like T-duality correspondences, but globally are glued together in a way that is  mathematically not yet fully understood. Exploring this gluing process is the motivation for the present paper, in which we thus study the \emph{automorphisms} of T-duality correspondences.

In their seminal work \cite{bunke2005a,bunke2006a} about topological T-duality, Bunke, Rumpf, and Schick constructed
a classifying space $B_n$ for  T-duality correspondences  (called T-duality triples there), i.e., they proved a bijection between homotopy classes of maps $X \to B_n$ and equivalences classes of topological T-duality correspondences over $X$. Matching physicists expectation, they also obtained an action by homotopy equivalences of the group  $\og{n,n,\Z}$ of integral orthogonal matrices w.r.t. the split-signature metric on $\R^{2n}$ on the classifying space $B_n$ \cite[Prop. 2.18]{bunke2006a}. While this was very successful, e.g. for studying  T-dualizability,  ongoing work requires to refine the work of Bunke et al. and to pass from  the level of homotopy classes and homotopy equivalences to a  \quot{homotopy-coherent} picture. For example, in the original purpose of T-duality, i.e., modelling dualities between string backgrounds,  total spaces $E$ and $\hat E$ are equipped with metrics, and  the cohomology classes $\xi$ and $\hat\xi$ actually live in \emph{differential} cohomology -- both structures are not homotopy invariants. Another reason, related to the motivation of the present article,  is the objective to glue T-duality correspondences -- it requires to understand them as objects in a (higher) category, and to study the automorphisms of these objects.

In my joint work with Nikolaus \cite{Nikolause} we obtained such refinement of the classifying space $B_n$: we proved that  $B_n$ is the geometric realization of  a Lie 2-groupoid; more precisely, it is the geometric realization of the delooping of a finite-dimensional strict Lie 2-group $\TD[n]$, i.e., $B_n \cong |B\TD[n]|$, see \cite[Lem. 3.8]{Nikolause}. Our result exhibits topological T-duality correspondences as objects of the 2-stack represented by $B\mathbb{TD}_n$,  
\begin{equation}
\label{eq:eq}
\sheaf{B\TD}_n^{+} \cong \tcor_n\text{,}
\end{equation} 
correctly reflecting their full bicategorical structure. 

The T-duality 2-group $\TD[n]$, which encodes via \cref{eq:eq} \textit{all} information about topological T-duality, is itself a very simple object. Represented as a crossed module, it consists of the Lie groups $G:= \R^{2n}$ and $H:= \Z^{2n} \times \ueins$, of the Lie group homomorphism
\begin{equation*}
t:\Z^{2n} \times \ueins \to \R^{2n} : (m,s) \mapsto m\text{,}
\end{equation*}
and of the following action $\alpha: G \times H \to H$ of $\R^{2n}$ on $\Z^{2n} \times \ueins$:
\begin{equation*}
\alpha(a,(m,s)) := (m,s-[a,m])\text{.} 
\end{equation*} 
Here, we write $\ueins = \R/\Z$ additively, and $[-,-]:\R^{2n} \times \R^{2n} \to \R$ is the bilinear form with matrix 
\begin{equation*}
J:=\begin{pmatrix}
0 & 0 \\
E_n & 0 \\
\end{pmatrix}\text{,}
\end{equation*}
where $E_n\in \R^{n \times n}$ is the unit matrix of size $n$. The Lie 2-group $\TD[n]$ is an example of Ganter's categorical tori \cite{Ganter2014} and as such a central Lie 2-group extension
\begin{equation*}
1\to B\ueins \to \TD[n] \to \T^{2n}_{dis} \to 1\text{,}
\end{equation*}
where $B\ueins$ denotes the Lie 2-group on a single object  with automorphism group $\ueins$, and $(..)_{dis}$ is notation for   considering a Lie group  as a Lie 2-group with only identity morphisms. Under Schommer-Pries' classification of such extensions  \cite{pries2}, it corresponds to the class
\begin{equation*}
\sum_{i=1}^{n} \mathrm{c}_i \cup \mathrm{c}_{i+n} \in \h^4(B\T^{2n},\Z)\text{,}
\end{equation*}
where $\mathrm{c}_i\in \h^2(B\T^{2n},\Z)$ is the pullback of the first Chern class along $\pr_i: B\T^{2n} \to B\ueins$  \cite[Prop. 3.5]{Nikolause}.

 Bunke-Rumpf-Schick  also described another  operation on T-duality correspondences  that depends on an anti-symmetric matrix $B\in \mathfrak{so}(n,\Z)$ \cite[Thm. 2.24]{bunke2006a}. This operation fixes the left leg of a topological T-duality correspondence (up to isomorphism) and permutes all possible right legs that can be  topologically T-dual  to the fixed left leg.
In \cite[§4.1]{Nikolause} Nikolaus and I improved this operation to a homotopy-coherent action of the additive group $\mathfrak{so}(n,\Z)$ on the Lie 2-group $\TD[n]$ by  Lie 2-group homomorphisms. We used this to establish a higher geometric approach to \quot{missing/non-geometric T-duals}, alternative to their famous treatment via non-commutative geometry by Mathai-Rosenberg \cite{Mathai2005a,Mathai2006a}.

Above, by \quot{homotopy-coherent action} of a group $G$ on the Lie 2-group $\TD[n]$ we mean that attached to each group element $g\in G$ is a Lie 2-group homomorphism $F_g: \TD[n] \to \TD[n]$, attached to each pair $g_1,g_2\in G$ is a 2-isomorphism $F_{g_2} \circ F_{g_1} \cong F_{g_2g_1}$, and that these 2-isomorphisms satisfy a coherence condition for any triple of group elements. Equivalently, a homotopy-coherent action can be described as a 2-group homomorphism $F:G_{dis} \to \AUT(\TD[n])$ to the automorphism 2-group of $\TD[n]$.
We remark that any such 2-group homomorphism induces, on the level of isomorphism classes, an ordinary group homomorphism $\pi_0 F:G \to \pi_0\AUT(\TD[n])$, which, under geometric realization, gives precisely an action of $G$   on the classifying space $B_n$ by homotopy equivalences. Let us  also clarify at this point that we always consider the most general  form of automorphisms of Lie 2-groups, which are  -- depending on the context -- known   as weak equivalences, butterflies, Hilsum-Skandalis maps, or anafunctors.

The present article emerged from the attempt to lift, in a similar way, Bunke-Rumpf-Schick's action of $\mathrm{O}(n,n,\Z)$ on $B_n$ by homotopy equivalences to a homotopy-coherent action on  $\TD[n]$. The failure of all our attempts to do so imposes the task to study the automorphism 2-group $\AUT(\TD[n])$ and its relation to the group $\mathrm{O}(n,n,\Z)$ in more generality. The main result of this article is a fairly complete and satisfying solution, see \cref{th:main} below.

One of the first results  we obtain from studying $\AUT(\TD[n])$ is that  not the group $\mathrm{O}(n,n,\Z)$ is relevant but rather a larger group that contains $\mathrm{O}(n,n,\Z)$ as a subgroup of index two. This is the  \textit{pseudo}-orthogonal group $\mathrm{O}^{\pm}(n,n,\Z)$ consisting of elements of $\mathrm{GL}(2n,\Z)$ that preserve the split-signature metric  \textit{up to a sign}. The group $\mathrm{O}^{\pm}(n,n,\Z)$ already appeared in \cite{Mathai2006} in the context of T-duality; here we manifest its relevance by obtaining it as the result of a computation.   

Also of importance are several interesting subgroups of $\mathrm{O}^{\pm}(n,n,\Z)$, which have been described already in the early literature about T-duality, e.g. \cite{Giveon1994}. Most importantly, the group $\mathfrak{so}(n,\Z)$ that appeared above is  a subgroup of $\mathrm{O}(n,n,\Z)$, via the embedding 
\begin{equation*}
B \mapsto \begin{pmatrix}
E_n & 0 \\
B & E_n \\
\end{pmatrix}\text{.}
\end{equation*}
Further, there is a  
$\mathrm{GL}(n,\Z)$-subgroup of \quot{coordinate changes}, a subgroup $V \cong (\Z/2\Z)^{n}$, whose $i$-th factor corresponds to  the  matrix that permutes the $i$-th and the $(n+i)$-th entry of vectors in $\Z^{2n}$, and a Klein 4-subgroup $K=\Z/2\Z \times \Z/2\Z$.
The following is the main result of this article.

\begin{theorem}
\label{th:main}
The automorphism 2-group of the T-duality 2-group is a non-central 2-group extension
\begin{equation*}
1\to B\Z^{2n} \to \AUT(\TD[n]) \to \mathrm{O}^{\pm}(n,n,\Z)_{dis}\to 1\text{.}
\end{equation*} 
Moreover, it is 2-torsion in the group of equivalence classes of such extensions. Finally, the extension splits canonically in a coherent multiplicative way if either $n=1$ or when restricted to any of the subgroups $\mathfrak{so}(n,\Z)$, $\mathrm{GL}(n,\Z)$, $V$, or $K$ of $\mathrm{O}^{\pm}(n,n,\Z)$.
\end{theorem}

\noindent
We prove \cref{th:main} in the following way:
\begin{enumerate}[(a)]

\item 
In \cref{sec:lie2grps,sec:automorphisms} we set up a small model $\AUT_{CI}(\Gamma)$ for automorphisms of general strict Lie 2-groups $\Gamma$ (\quot{crossed intertwiners}) and prove in \cref{th:auto} that this small model is equivalent to the full automorphism 2-group $\AUT(\Gamma)$ for a certain class of strict Lie 2-groups, to which the T-duality 2-group $\TD[n]$ belongs.

\item
In \cref{sec:auttdual} we construct explicitly a strict Lie 2-group $\AU[n]$ that is suitable for calculations, and   show in \cref{prop:modelCI} that  it is equivalent  to our small model for the automorphisms of $\TD[n]$, i.e., $\AU[n]\cong\AUT_{CI}(\TD[n])$.

\end{enumerate}
All properties of $\AUT(\TD[n])$ claimed in \cref{th:main} are then
proved performing calculations with $\AU[n]$  in \cref{sec:classification}. These are, to the authors regret, \quot{preposterous  calculations with matrices}.  In \cref{prop:pi0,prop:pi1} we compute the homotopy groups of  $\AU[n]$, and conclude in \cref{th:extension} that  it is a non-central extension as claimed in \cref{th:main}. In \cref{sec:splitting:multiplicativity} we compute the k-invariant of $\AU[n]$, which is  an element  in the group cohomology
\begin{equation*}
\h^3(\mathrm{O}^{\pm}(n,n,\Z),\Z^{2n})
\end{equation*}
that classifies 2-group extensions of $\mathrm{O}^{\pm}(n,n,\Z)$ by $\Z^{2n}$, representing it by an explicit  group cocycle $m: \mathrm{O}^{\pm}(n,n,\Z)^3 \to \Z^{2n}$, see \cref{prop:coex}. We then show the following, completing the proof of \cref{th:main}:
\begin{enumerate}[(a)]

\item 
We have $m=0$ for $n=1$ (\cref{prop:n1}) or over the subgroup $V\cong (\Z/2\Z)^{n}$ (\cref{ex:coherence:I}), or over the subgroups $\mathrm{GL}(n,\Z)$, $\mathfrak{so}(n,\Z)$, or $K$ (\cref{lem:resglso}).

\item
$2m$ is the coboundary of a canonical cochain $\gamma$ (\Cref{prop:torsion}). 
\end{enumerate}

\bigskip

Let us describe the consequences of \cref{th:main} for the various actions and types of actions that motivated this work. First of all, by saying that the extension splits canonically in a coherent multiplicative way we mean that we provide  2-group homomorphisms
\begin{equation*}
\mathrm{O}^{\pm}(1,1,\Z)_{dis} \to \AUT(\TD[1])
\quand
G_{dis} \to  \AUT(\TD[n])
\end{equation*}
for $G \in \{\mathfrak{so}(n,\Z),\mathrm{GL}(n,\Z),V,K\}$, which are  sections against the projection of the 2-group extension in  \cref{th:main}; see \cref{prop:actions}.  Along these 2-group homomorphisms we induce homotopy-coherent actions of the groups $\mathrm{O}^{\pm}(1,1,\Z)_{dis}$ and $G$  on $\TD[n]$.  We show that in case of $\mathfrak{so}(n,\Z)$, this is precisely the action  discovered manually in \cite{Nikolause}, see \cref{prop:inducedaction}. We include at the end of this article  a discussion (\cref{sec:Tautomorphisms}) about the geometric meaning of these group actions on T-duality correspondences. 

\cref{th:main} implies that
\begin{equation*}
\pi_1\AUT(\TD[n])\cong \Z^{2n}
\quand
\pi_0\AUT(\TD[n]) \cong \mathrm{O}^{\pm}(n,n,\Z)\text{.}
\end{equation*} 
The inverse of the second isomorphism, 
\begin{equation*}
\mathrm{O}^{\pm}(n,n,\Z) \to \pi_0\AUT(\TD[n])
\end{equation*}
is described explicitly in \cref{sec:homtopycrossedmodule}; it reproduces the action of Bunke-Rumpf-Schick by homotopy equivalences. The question if this action can be lifted to a homotopy-coherent action on $\TD[n]$ is equivalent to the question whether or not the  extension in \cref{th:main} splits as a whole in a coherent multiplicative way. 
We conjecture that this is not the case, though we do not have a definite answer to this question.
Apart from the fact that we tried and failed to find such splitting, it is a fact that the group cohomology of $\mathrm{GL}(2n,\Z)$  indeed has 2-torsion  (see the discussion at  \cite{Chapdelaine}), though neither $\h^3(\mathrm{GL}(2n,\Z),\Z)$, let alone $\h^3(\mathrm{O}^{\pm}(n,n,\Z),\Z)$ seem to be known or easily computable. Our conjecture contradicts the physicists expectation that the group $\mathrm{O}^{\pm}(n,n,\Z)$ can be used for gluing T-duality correspondences to obtain T-folds. Instead, it suggests that  the full 2-group $\AUT(\TD_n)$ must be used; this will be the subject of further work.

\paragraph{Acknowledgements. } I would like to thank Thomas Nikolaus for our collaboration that initiated this project, and Nora Ganter and Christian Saemann for helpful discussions on several aspects.
I would also like to thank a referee for their comments and suggestions that improved this article considerably. 

\setsecnumdepth{2}

\section{Lie 2-groups and  homomorphisms}

\label{sec:lie2grps}

We recall and describe two perspectives on strict Lie 2-groups: the first is via crossed modules, and the second is via Lie groupoids. Both perspectives have their advantages and disadvantages, and we will later switch back and forth whenever convenient. Crossed modules are often convenient for computations, as they form a kind of minimal way to represent Lie 2-groups. Lie groupoids fit better into the abstract concept of categorified groups, and often provide a more conceptional point of view.   

\subsection{Crossed modules and crossed intertwiners}

\label{sec:strictint}

A \emph{crossed module} of Lie groups is a quadruple $\Gamma=(G,H,t,\alpha)$ consisting of Lie groups $G$ and $H$, of a Lie group homomorphism $t:H \to G$, and of a smooth action $\alpha:G \times H \to H$ of $G$ on $H$ by group homomorphisms, such that $\alpha(t(h),h')=hh'h^{-1}$ and $t(\alpha(g,h))=gt(h)g^{-1}$
hold for all $g\in G$ and $h,h'\in H$. 

Let $\Gamma=(G,H,t,\alpha)$ and $\Gamma'=(G',H',t',\alpha')$ be crossed modules of Lie groups. The easiest form of a homomorphism between crossed modules is a \emph{strict intertwiner}: a pair $(\phi,f)$ consisting of Lie group homomorphisms $\phi:G \to G'$ and $f:H \to H'$ such that
\begin{enumerate}[(a)]

\item 
$\phi(t(h))=t'(f(h))$.

\item
$f(\alpha(g,h))=\alpha'(\phi(g),f(h))$ 

\end{enumerate}
hold for all $g\in G$ and $h\in H$. There is an obvious associative composition of intertwiners, as well as an identity intertwiner. We thus obtain a category $\CMI$.

Suppose $\Gamma=(G,H,t,\alpha)$ is a crossed module. Since $t(H)\subset G$ is normal, we may consider the quotient group $G/t(H)$. We also  consider the Lie group  $U := \mathrm{ker}(t)\subset H$. It has the following well-known properties:
\begin{itemize}

\item 
$U$ is  abelian.

\item
$U$ is central in $H$.

\item
The action of $G$ on $H$ restricts to $U$.

\item
The induced action of $H$ on $U$ via $t: H \to G$ is trivial.

\item
The action of $G$ on $U$ descends to $G/t(H)$.

\end{itemize}
\label{sec:CI}
Using the group $U$ associated to any crossed module, we discuss now a  weaker version of homomorphisms between crossed modules, which appeared already in the context of T-duality, see \cite[§A]{Nikolause}. In the present article, this weaker version turns out to be suitable for modelling the automorphism 2-group of $\TD[n]$. 

\begin{definition}
A \emph{crossed intertwiner} between smooth crossed modules $\Gamma=(G,H,t,\alpha)$ and $\Gamma'=(G',H',t',\alpha')$ is a triple $(\phi,f,\eta)$ consisting of Lie group homomorphisms $\phi: G \to G'$ and $f: H \to H'$, and of a smooth map $\eta: G \times G \to U'$ satisfying for all $h,h'\in H$ and $g,g',g''\in G$ the following axioms:
\begin{enumerate}[({CI}1),leftmargin=*]

\item
\label{CI1}
$\phi(t(h))=t(f(h))$.

\item
\label{CI2}
$\eta(t(h),t(h')) =1$

\item
\label{CI4}
$\eta(g,t(h)g^{-1})\cdot f(\alpha(g,h))=\alpha'(\phi(g),\eta(t(h)g^{-1},g))\cdot \alpha'(\phi(g), f(h))$.

\item
\label{CI5}
$\eta(g,g')\cdot \eta(gg',g'')=\alpha'(\phi(g),\eta(g',g''))\cdot \eta(g,g'g'')\text{.}$
\end{enumerate}
\end{definition}

\noindent
We remark that these axioms imply the following:
\begin{itemize}

\item
$f(u)\in U'$ for all $u\in U$, i.e., $f$ induces a homomorphism $U \to U'$.

\item
$\eta(g,1)=1=\eta(1,g)$.

\item
$\eta(g,g^{-1})=\alpha'(\phi(g),\eta(g^{-1},g))$.

\end{itemize}
Strict intertwiners are precisely the crossed intertwiners with $\eta=1$. 
\label{sec:crossedtrans}

\begin{definition}
Let $(\phi,f,\eta)$ and $(\phi,f',\eta')$ be crossed intertwiners. A \emph{crossed transformation}
\begin{equation*}
\beta:(\phi,f,\eta) \Rightarrow (\phi,f',\eta')
\end{equation*}
is a smooth map $\beta: G \to U$ satisfying
\begin{enumerate}[({CT}1),leftmargin=*]

\item
\label{CT1}
$\beta(t(h))\cdot f(h)=f'(h) $

\item
\label{CT2}
$\beta(g_1)\alpha(\phi(g_1),\beta(g_2))\eta(g_1,g_2) = \eta'(g_1,g_2)  \beta
(g_1g_2)$. 

\end{enumerate}
\end{definition}

The  composition $\beta'\bullet\beta$ of $\beta: (\phi,f,\eta) \Rightarrow (\phi,f',\eta')$ and $\beta': (\phi,f',\eta') \Rightarrow (\phi,f',\eta'')$ is given by the smooth map $(\beta'\bullet\beta)(g) := \beta'(g)\beta(g)$. The identity transformation is $\beta := 1$. Hence, crossed intertwiners and crossed transformation form a groupoid, which we denote by $\CI(\Gamma,\Gamma')$.

Next we define a composition functor
\begin{equation*}
\CI(\Gamma',\Gamma'') \times \CI(\Gamma,\Gamma') \to \CI(\Gamma,\Gamma'')\text{.}
\end{equation*}
On the level of objects, the composition of crossed intertwiners is defined by
\begin{equation*}
(\phi_2,f_2,\eta_2)\circ (\phi_1,f_1,\eta_1):=(\phi_2 \circ \phi_1, f_2 \circ f_1,\eta_2 \circ (\phi_1 \times \phi_1) \cdot f_2 \circ \eta_1 )
\end{equation*}
On the level of morphisms, we define the following \quot{horizontal} composition 
\begin{equation*}
\beta_2 \circ \beta_1 :(\phi_2,f_2,\eta_2) \circ (\phi_1,f_1,\eta_1)  \Rightarrow (\phi_2,f_2',\eta_2') \circ (\phi_1,f_1',\eta_1') 
\end{equation*}
of crossed transformations $\beta_1: (\phi_1,f_1,\eta_1) \Rightarrow (\phi_1,f_1',\eta_1')$ and $\beta_2: (\phi_2,f_2,\eta_2) \Rightarrow (\phi_2,f_2',\eta_2')$, given by
\begin{equation*}
(\beta_2 \circ \beta_1)(g) := \beta_2(\phi_1(g))\cdot f_2(\beta_1(g))\text{.} \end{equation*}
It is easy to check that composition is strictly associative.
The crossed homomorphism  $(\id_G,\id_H,1)$ is strictly neutral with respect to composition.
All together, we obtain a strict bicategory $\CM$  of crossed modules, crossed intertwiners, and crossed transformations. The former category $\CMI$  embeds via a faithful functor $\CMI \to \CM$.

The following is straightforward to check.

\begin{lemma}
A crossed intertwiner $(\phi,f,\eta)$ is invertible if and only if $\phi$ and $f$ are Lie group isomorphisms; its inverse is $(\phi^{-1},f^{-1},\eta^{-1})$, where
\begin{equation*}
\eta^{-1}(g,g'):= f^{-1}( \eta(\phi^{-1}(g),\phi^{-1}(g')))^{-1}\text{.}
\end{equation*}
\end{lemma}

It is straightforward to check that the groups $G/t(H)$ and $U$, as well as the action of $G/t(H)$ on $U$ are invariant under invertible crossed intertwiners.

\begin{example}
\label{ex:flip}
In case of the crossed module $\TD_n$ (see \cref{sec:intro} or \cref{def-TDn}) we consider a certain crossed intertwiner
\begin{equation*}
flip: \TD[n] \to \TD[n]\text{.}
\end{equation*}
We will see later in \cref{sec:actions} that it swaps the left and the right leg of  a T-duality correspondence. It is given by $\phi(x \oplus \hat x) :=\hat x \oplus x$ and $f(m \oplus \hat m,t) := (\hat m \oplus m,t)$ and $\eta(a,a') :=(0,[a,a'])$, where the bracket $[-,-]$ was introduced in \cref{sec:intro}. Here, $x,\hat x\in \R^{n}$ and  $x\oplus \hat x \in \R^{2n}$, similarly for $m,\hat m\in \Z^{n}$.
Note that the flip is not strictly involutive: the composition $flip^2=flip \circ flip$ is the crossed intertwiner $(\id,\id,\tilde\eta)$ with
\begin{align*}
\tilde\eta(x\oplus \hat x,x' \oplus \hat x') &=\eta (\phi(x\oplus \hat x), \phi(x' \oplus \hat x')) + f(\eta(x\oplus \hat x,x' \oplus \hat x')) 
\\& =[\hat x \oplus x,\hat x'\oplus x']+[x\oplus \hat x,x' \oplus \hat x']
\\& = x\hat x'+ \hat xx' \text{.}
\end{align*}
We define a crossed transformation $\beta: flip^2 \Rightarrow \id$ by setting
$\beta(x \oplus \hat x):=-x \hat x$ (standard scalar product on $\R^{n}$).
This crossed transformation satisfies a coherence law that will be discussed in  \cref{ex:flipautomorphism}. 
\end{example}

\begin{remark}
Weak equivalences between   crossed modules can be modelled by so-called \emph{butterflies} \cite{Aldrovandi2009}, certain group extensions $H' \to K \to G$. Crossed intertwiners correspond precisely to those butterflies whose extension splits, i.e.,  those that have a section by a Lie group homomorphism. 
\end{remark}

\subsection{Lie 2-groups and semi-strict homomorphisms}

A \emph{strict Lie 2-group} is a Lie groupoid $\Gamma$ whose manifolds of objects $\Gamma_0$ and morphisms $\Gamma_1$ are Lie groups, in such a way that source, target,  $\id: \Gamma_0 \to \Gamma_1$, the composition, and the inversion (w.r.t. composition) are group homomorphisms. A \textit{strict homomorphism} between strict Lie 2-groups is a smooth functor $F: \Gamma \to \Gamma'$ that is a group homomorphism both on the level of objects and on the level of morphisms.

A crossed module $\Gamma=(G,H,t,\alpha)$ determines a strict Lie 2-group, intentionally denoted by the same letter $\Gamma$,
by putting $\Gamma_0 := G$ and $\Gamma_1 := H \ltimes_{\alpha} G$, source $(h,g) \mapsto g$, target $(h,g) \mapsto t(h)g$, and composition $(h_2,g_2)\circ(h_1,g_1) := (h_2h_1,g_1)$. Likewise, an intertwiner $(\phi,f):\Gamma \to \Gamma'$ induces a strict homomorphism $F:\Gamma \to \Gamma'$ defined by $F(g):=\phi(g)$ and $F(h,g) := (f(h),\phi(g))$. 
It is well-known that this establishes an equivalence of categories
\begin{equation*}
\CMI \cong \sL2Gs
\end{equation*}
between smooth crossed modules and intertwiners, and strict Lie 2-groups and strict Lie 2-group homomorphisms.
Under this equivalence, the quotient group $G/t(H)$ is the group $\pi_0\Gamma$ of isomorphism classes of objects of the Lie groupoid $\Gamma$, and  $U=\pi_1\Gamma:=\mathrm{Aut}(1)$, the abelian group of automorphisms of the unit object $1\in \Gamma_0$. This gives a clear interpretation of these groups as the usual homotopy groups of the 1-type $\Gamma$.
Note that $\pi_0\Gamma$ acts on $\pi_1\Gamma$.

A crossed intertwiner $(\phi,f,\eta)$ defines a smooth functor
$F: \Gamma \to \Gamma'$
by putting 
\begin{equation*}
F(g):=\phi(g)
\quand
F(h,g):= (\eta(t(h),g)^{-1}\cdot f(h),\phi(g))\text{.}
\end{equation*}
The functor $F$ is a group homomorphism on the level of objects, but not on the level of morphisms, and thus is not a morphism in $\sL2Gs$. In this context it is  more intuitive to view the group structures on $\Gamma_0$ and $\Gamma_1$ as equipping the Lie groupoid $\Gamma$ with a strict monoidal structure $m: \Gamma \times \Gamma \to \Gamma$. Then, the functor $F:\Gamma \to \Gamma'$ is  monoidal functor; however, it is not \textit{strictly} monoidal: this is expressed by the existence of a smooth natural transformation (\quot{multiplicator})
\begin{equation*}
\alxydim{}{\Gamma \times \Gamma \ar[r]^-{m} \ar[d]_{F \times F} & \Gamma \ar@{=>}[dl]|*+{\tilde\eta} \ar[d]^{F} \\ \Gamma' \times \Gamma' \ar[r]_-{m} & \Gamma'\text{.}}
\end{equation*}
It is defined from the given map $\eta$, and its components are
$\tilde\eta(g,g') := (\eta(g,g'),\phi(gg'))$.
Moreover, $\tilde\eta$ satisfies the usual coherence law for the multiplicators of monoidal functors, which corresponds on the level of components to the equality
\begin{equation}
\label{eq:cohetatilde}
(F(\id_{g_{34}}) \cdot \tilde\eta(g_{23},g_{12})) \circ \tilde \eta(g_{34},g_{23}g_{12})
= ( \tilde\eta(g_{34},g_{23}) \cdot F(\id_{g_{12}})) \circ  \tilde \eta(g_{34}g_{23},g_{12})
\end{equation} 
for $g_{12},g_{23},g_{34} \in G$, which can easily be deduced from \cref{CI5*}.

Associated to a crossed transformation $\beta:(\phi,f,\eta) \Rightarrow (\phi,f',\eta')$  is a smooth natural transformation $\lambda: F \Rightarrow F'$, where $F$ and $F'$ are the smooth functors associated to $(\phi,f,\eta)$ and $(\phi,f',\eta')$, respectively. It is defined by $\lambda(g) := (\beta(g),\phi(g))$. 
Moreover, it is a monoidal natural transformation in the usual sense, i.e., we have an equality
\begin{equation*}
\alxydim{@C=3em@R=3em}{\Gamma \times \Gamma \ar@/_3pc/[d]_{F \times F}="1" \ar[r]^-{m} \ar[d]^{F' \times F'}="2" \ar@{=>}"1";"2"|{\lambda \times \lambda} & \Gamma  \ar@{=>}[dl]|<<<<<<<*+{\tilde\eta'} \ar[d]^{F'}  \\ \Gamma' \times \Gamma' \ar[r]_-{m} & \Gamma'}
=
\alxydim{@C=3em@R=3em}{\Gamma \times \Gamma  \ar[r]^-{m} \ar[d]_{F \times F} & \Gamma \ar@/^3pc/[d]^{F'}="2" \ar@{=>}[dl]|*+{\tilde\eta} \ar[d]_{F}="1" \ar@{=>}"1";"2"|{\lambda} \\ \Gamma' \times \Gamma' \ar[r]_-{m} & \Gamma'}
\end{equation*}
Finally, the vertical composition of crossed transformations corresponds precisely to the composition of monoidal natural transformation.

Let $\ssL2gh(\Gamma,\Gamma')$ denote the category of semi-strict Lie 2-group homomorphisms, i.e. smooth functors $F: \Gamma \to \Gamma'$ that are monoidal by virtue of a smooth multiplicator $\tilde\mu$, and smooth natural transformations $\lambda: F \Rightarrow F'$. Above we have defined a functor
\begin{equation}
\label{eq:crinttosestr}
\CI(\Gamma,\Gamma') \to \ssL2gh(\Gamma,\Gamma')\text{.}
\end{equation}
We remark that this functor is faithful, but in general neither essentially surjective nor full.  

It is straightforward to check that the functor associated to a composition of crossed intertwiners is the composition of the separate functors, i.e.
\begin{equation}
\label{eq:compfunctors}
F_{(\phi_2,f_2,\eta_2)\circ (\phi_1,f_1,\eta_1)}=F_{(\phi_2,f_2,\eta_2)} \circ F_{(\phi_1,f_1,\eta_1)}\text{.}
\end{equation}
Furthermore, the \quot{stacking}
\begin{equation*}
\alxydim{}{\Gamma \times \Gamma \ar[r]^-{m} \ar[d]_{F_1 \times F_1} & \Gamma \ar@{=>}[dl]|{\tilde\eta_1} \ar[d]^{F_1} \\ \Gamma' \times \Gamma' \ar[d]_{F_2 \times F_2} \ar[r]|--{m'} & \Gamma' \ar@{=>}[dl]|{\tilde\eta_2} \ar[d]^{F_2} \\ \Gamma'' \times \Gamma'' \ar[r]_-{m''} & \Gamma''}
\end{equation*}
of the corresponding multiplicators $\tilde\eta_1$ and $\tilde\eta_2$ is precisely the multiplicator of the composition.
This shows that \cref{eq:compfunctors} is an equality of monoidal functors. 

We are now in position to form the bicategory $\L2Gs$ whose objects are strict Lie 2-groups and whose morphism categories are $\ssL2gh(\Gamma,\Gamma')$. We have then constructed a 2-functor
\begin{equation*}
\CM \to \L2Gs\text{.}
\end{equation*}

\setsecnumdepth{1}

\section{The automorphism 2-group of a Lie 2-group}

\label{sec:automorphisms}

Let $\Gamma$ be a strict Lie 2-group. There are various versions of automorphisms one can consider, all forming 2-groups. The most general one is a semi-strict 2-group $\AUT(\Gamma)$ defined in the following way:  
\begin{itemize}

\item 
Objects are\ all weak equivalences of Lie 2-groups, i.e. invertible smooth anafunctors $F: \Gamma \to \Gamma$  together with invertible smooth anafunctor transformations
\begin{equation*}
\alxydim{}{\Gamma \times \Gamma \ar[r]^-{m} \ar[d]_{F \times F} & \Gamma \ar@{=>}[dl]|*+{\mu} \ar[d]^{F} \\ \Gamma \times \Gamma \ar[r]_-{m} & \Gamma\text{;}}
\end{equation*}
satisfying a coherence condition. Multiplication is the composition of anafunctors.

\item
Morphisms are all anafunctor transformations $\lambda: F \Rightarrow F'$ that are compatible with the transformations $\mu$. Composition is the vertical composition of transformations, and multiplication is the horizontal composition of transformations. 

\end{itemize}
This yields a \emph{semi-strict} 2-group since the composition of anafunctors is associative only up to coherent transformations. A reference for smooth anafunctors is \cite[§2.3]{Nikolaus}. They are also known as bibundles, or Hilsum-Skandalis morphisms. In terms of crossed modules, weak equivalences correspond to so-called \textit{butterflies} \cite{Aldrovandi2009}.

A more restrictive, however strict 2-group $\AUT^{fun}(\Gamma)$ is obtained by replacing \quot{anafunctor} by \quot{smooth functor}, and \quot{anafunctor transformation} by \quot{smooth natural transformation}.  In other words, $\AUT^{fun}(\Gamma)\subset \ssL2gh(\Gamma,\Gamma)$ is the subgroupoid consisting of all invertible semi-strict Lie 2-group homomorphisms, and all invertible natural transformations.   
The  fully faithful inclusion  $\fun^{\infty}(\mathcal{X},\mathcal{Y})\subset \mathcal{A}na^{\infty}(\mathcal{X},\mathcal{Y})$ of smooth functors into smooth anafunctors induces a fully faithful monoidal functor $\AUT^{fun}(\Gamma) \to \AUT(\Gamma)$.

The next version relies on the crossed module  that corresponds to $\Gamma$. It is the strict 2-group  $\AUT_{CI}(\Gamma) \subset \CI(\Gamma,\Gamma)$ consisting of all invertible crossed  intertwiners and all crossed transformations. The functor \cref{eq:crinttosestr} restricts to a monoidal functor $\AUT_{CI}(\Gamma) \to \AUT^{fun}(\Gamma)$. The main point of this section is to show that the smallest of these automorphism 2-groups, $\AUT_{CI}(\Gamma)$, still captures the full automorphism group $\AUT(\Gamma)$ is certain cases. 

\begin{theorem}
\label{th:auto}
Let $\Gamma=(G,H,t,\alpha)$ be a crossed module with the following properties:
\begin{enumerate}[({P}1),leftmargin=*]

\item
\label{aut:1} 
$G$ is contractible.

\item
\label{aut:2}
 $Z := \mathrm{im}(t) \subset G$ is a discrete, free abelian subgroup. 
 
\item
\label{aut:3}
$U :=\mathrm{ker}(t)$ is connected.

\end{enumerate}
Then, the functors
\begin{equation*}
\AUT_{CI}(\Gamma) \to \AUT^{fun}(\Gamma) \to \AUT(\Gamma)\text{.}
\end{equation*}
are isomorphisms of 2-groups. 
\end{theorem}

\begin{proof}
Split into \cref{lem:aut:step1,lem:aut:step2,lem:aut:step3}. 
\end{proof}

\begin{remark}
Of course, assumptions \cref{aut:1*,aut:2*,aut:3*} are  very restrictive. We recall \cite[Theorem 3.2]{Onishchik1991} that \cref{aut:1*} implies that $G\cong \R^{n}$ as Lie groups. Note that this is the case  for $\TD[n]$. 
\end{remark}

\begin{example}
Let $U$ be a connected abelian Lie group. We consider the crossed module with $G$ the trivial group, $H:= U$, and $t$ and $\alpha$ trivial. It satisfies all conditions of \cref{th:auto}. The corresponding Lie 2-group is $\Gamma = BU$, i.e. it has a single object with automorphism group $U$. It is easy to see that 2-group $\AUT_{CI}(\Gamma)$ is the ordinary group $\mathrm{Aut}(U)$ of Lie group automorphisms of $U$ (with only identity morphisms). Hence, by \cref{th:auto}, this is also the full automorphism group.  

\end{example}

\begin{proposition}
\label{lem:aut:step1}
Under assumption \cref{aut:1*}  the functor 
\begin{equation*}
\AUT^{fun}(\Gamma) \to \AUT(\Gamma)
\end{equation*}
is an equivalence of categories.
\end{proposition}

\begin{proof}
It remains to show that it is essentially surjective.
If $F:\Gamma \to \Gamma$ is a smooth anafunctor,  then its left anchor
$\alpha_l:F \to G$ is a principal $\Gamma$-bundle over a  contractible manifold. In particular, it is an ordinary principal bundle (for the Lie group $H$), and hence has a global section. Hence, $F$ is equivalent to a smooth functor.
\end{proof}

It remains to show that the functor $\AUT_{CI}(\Gamma) \to \AUT^{fun}(\Gamma)$ is an equivalence of categories; this is done in the subsequent \cref{lem:aut:step2,lem:aut:step3}. Before we start we give a reformulation of objects and morphisms of  $\AUT^{fun}(\Gamma)$ by spelling out what functors and natural transformations are in terms of the crossed module $(G,H,t,\alpha)$. 

Let $(F,\mu)$ be an object in $\AUT^{fun}(\Gamma)$.  We let $\phi:G \to G$ be the assignment of $F$ on objects,  extract $f:H \times G \to H$ on the level of morphisms (we ignore the $G$-component in the target), and let $\eta:G \times G \to H$  be the component map of the natural transformation $\mu$ (we ignore again the $G$-component).
These three maps have the following properties:
\begin{enumerate}[({CA}1),leftmargin=*]
\itemindent=2em

\item
\label{CA1}
$\phi(t(h)g)=t(f(h,g))\phi(g)$.

\item
\label{CA2}
$f(h'h,g)=f(h',t(h)g)\cdot f(h,g)$

\item 
\label{CA3}
$t(\eta(g,g'))\phi(gg')=\phi(g)\phi(g')$.

\item
\label{CA4}
$\eta(t(h)g,t(h')g') \cdot f(h\alpha(g,h'),gg')=f(h,g)\cdot \alpha(\phi(g),f(h',g'))\cdot \eta(g,g')$

\item
\label{CA5}
$\eta(g,g')\cdot \eta(gg',g'')=\alpha(\phi(g),\eta(g',g''))\cdot \eta(g,g'g'')\text{.}$
\end{enumerate}
Indeed,  \ref{CA1} and \ref{CA2} are satisfied because $F$ is a functor. \cref{CA3*} is the target condition for the natural transformation $\mu$, \cref{CA4*} is the naturality of $\mu$, and \cref{CA5*} is the coherence condition for $\mu$.
In turn, we have the following consequences of these properties:
\begin{enumerate}[({CA}1{a}),leftmargin=*]
\itemindent=2em
\setcounter{enumi}{1}

\item
\label{lem:funautnorm:a}
$f(1,g)=1$.

\item
\label{lem:funautnorm:b}
$t(\eta(1,1))=\phi(1)$.

\end{enumerate}
Indeed, \cref{lem:funautnorm:a*} follows from \ref{CA2} by putting $h'=h=1$. \cref{lem:funautnorm:b*} follows by putting $g=g'=1$ in \ref{CA3}. 

In terms of above reformulation, two triples $(\phi,f,\eta)$ and $(\phi',f',\eta')$ are isomorphic as objects of $\AUT^{fun}(\Gamma)$ if and only if  there exists a smooth map $\beta:G \to H$ such that
\begin{enumerate}[({CA}1),leftmargin=*]
\itemindent=2em
\setcounter{enumi}{5}

\item 
\label{CA6}
$t(\beta(g))\phi(g)=\phi'(g)$. 

\item
\label{CA7}
$\beta(t(h)g)f(h,g)=f'(h,g)  \beta(g)$.

\item
\label{CA8}
$\beta(g_1)\alpha(\phi(g_1),\beta(g_2))\eta(g_1,g_2) = \eta'(g_1,g_2)  \beta
(g_1g_2)$. 

\end{enumerate}
Indeed, if $\beta$ is the component map of a 1-morphism $\lambda:F \Rightarrow F'$, then \cref{CA6*} is the target condition for $\lambda$, \cref{CA7*} is the naturality, and \cref{CA8*} is the coherence condition between $\lambda$, $\mu$, and $\mu'$.

An object $(\phi,f,\eta)$ in $\AUT^{fun}(\Gamma)$ is called \textit{normalized}, if $\phi(1)=1$ and $\eta(1,1)=1$.  We have the following result:

\begin{lemma}
\label{lem:th:aut:2}
\begin{enumerate}[(a)]

\item 
\label{lem:th:aut:2:a}
Every object in $\AUT^{fun}(\Gamma)$ is isomorphic to a normalized one. 

\item
\label{lem:th:aut:2:b}
If $(\phi,f,\eta)$ is normalized, then we have
\begin{enumerate}[({CA-N}1),leftmargin=*]
\itemindent=2em

\item 
\label{lem:funautnorm:c}
$\eta(1,g)=\eta(g,1)=1$.

\item
\label{lem:funautnorm:d}
$f(h,g)=\eta(t(h),g)^{-1} \cdot f(h,1)$.

\end{enumerate}

\end{enumerate}
\end{lemma}
  
\begin{proof}
For \cref{lem:th:aut:2:a*} consider an object $(\phi,f,\eta)$. We define $\beta(g) := \eta(1,1)^{-1}$. Using \ref{CA6}, \ref{CA7}  and \ref{CA8} as definitions of $\phi'$, $f'$, and $\eta'$, we obtain an isomorphic object $(\phi',f',\eta')$ that is normalized.
For \cref{lem:funautnorm:c*} we put $g=g'=1$ in \ref{CA5} and obtain 
\begin{equation*}
1=\eta(1,1)=\alpha(\phi_0(1),\eta(1,g))=\alpha(t(\eta(1,1)),\eta(1,g)) = \eta(1,1) \cdot \eta(1,g) \cdot \eta(1,1)^{-1}=\eta(1,g)\text{,}
\end{equation*}
and thus $\eta(1,g)=1$. Putting $g'=g''=1$ in  \ref{CA5}  we get
$\eta(g,1)=1$.
We get \cref{lem:funautnorm:d*} by putting  $h'=1$ and $g=1$ in \ref{CA4}. 
\end{proof}  

An object $(\phi,f,\eta)$ in $\AUT^{fun}(\Gamma)$ is called \emph{UZ-normalized} if it is normalized, $\eta:G \times G \to H$ takes values in the abelian subgroup $U := \mathrm{ker}(t) \subset H$, and $\eta|_{Z \times Z}=1$. 

\begin{lemma}
\label{lem:th:aut:3}
\begin{enumerate}[(a)]

\item 
\label{lem:th:aut:3:a}
Under assumptions \cref{aut:1*,aut:2*,aut:3*}, every object in $\AUT^{fun}(\Gamma)$ is isomorphic to a $UZ$-normalized one. 

\item
\label{lem:th:aut:3:b}
If $(\phi,f,\eta)$ is UZ-normalized, then we have:
\begin{enumerate}[({CA-UZ}1),leftmargin=*]
\itemindent=2em

\item
\label{lem:funautnorm:f}
The map $f': H \to H$ with $f'(h) := f(h,1)$ is a group homomorphism.

\item 
\label{lem:funautnorm:e}
$\phi:G \to G$ is a group homomorphism.

\end{enumerate}

\end{enumerate}
\end{lemma}

\begin{proof}
For \cref{lem:th:aut:3:a*}, by \cref{lem:th:aut:2} it suffices to prove that every normalized object $(\phi,f,\eta)$ is isomorphic to a $UZ$-normalized one.  Since $G$ is a group and thus non-empty,  \cref{aut:1*} implies that $G$ is connected. Then, since $Z\subset G$ is discrete by \cref{aut:2*}, the smooth map $t \circ \eta:G \times G \to Z$ is constant, and by \cref{lem:funautnorm:c*} even constantly $1$. This means that $\eta$ takes values in $U$, i.e. $\eta:G \times G \to U$. Since $G$  acts on $U$ via $\alpha$, the subgroup $Z$ acts, too. Since $Z=\mathrm{im}(t)$, this action is trivial (for every crossed module).
Hence, the restriction of $\eta$ to $Z \times Z$ is a 2-cocycle on $Z$ with values in the trivial $Z$-module $U$, by \cref{CA5}. It classifies a central extension of $Z$ by $U$. Since $Z$ is free abelian, every central extension of $Z$ is trivial.
Hence, there exists a map $\beta_Z:Z \to U$ such that $\beta(z_1)\beta(z_2)\eta(z_1,z_2) =   \beta
(z_1z_2)$. We can assume that $\beta_Z(1)=1$. Since $Z \subset G$ is discrete and $U$ is connected by assumption \cref{aut:3*}, $\beta_Z$ can be extended to a smooth map $\beta:G \to U$ with $\beta|_Z=\beta_Z$. 
Using \ref{CA6}, \ref{CA7}  and \ref{CA8} as definitions of $\phi'$, $f'$, and $\eta'$, we obtain an isomorphic object $(\phi',f',\eta')$ which is  normalized (since $\beta(1)=1$), $\eta'$ still takes values in $U$, and additionally satisfies 
$\eta'(z_1,z_2)=1$
for all $z_1,z_2\in Z$. Hence, $(\phi,f,\eta)$ is UZ-normalized. Part \cref{lem:th:aut:3:b*} is trivial: \cref{lem:funautnorm:f*} follows from \cref{CA2,lem:funautnorm:d*}, 
and \cref{lem:funautnorm:e*} follows from \cref{CA3}.
\end{proof}

\begin{proposition}
\label{lem:aut:step2}
Under assumptions \cref{aut:1*,aut:2*,aut:3*}  the functor $\AUT_{CI}(\Gamma) \to \AUT^{fun}(\Gamma)$ is essentially surjective. 
\end{proposition}

\begin{proof}
Let $(\phi,f,\eta)$ be an object in $\AUT^{fun}(\Gamma)$. By \cref{lem:th:aut:3:a} we can assume that it is UZ-normalized. We claim that $(\phi,f',\eta)$ is a crossed intertwiner:
\begin{enumerate}[({CI}1),leftmargin=*]

\item 
is \cref{CA1*}.

\item
is satisfied because $\eta|_{Z \times Z}=1$.

\item
is proved by the following calculation:
\begin{align*}
\eta(g,t(h)g^{-1})\cdot f'(\alpha(g,h)) 
&=\eta(g,t(h)g^{-1})\cdot f(\alpha(g,h),1)
\\&\eqcref{CA4*} f(1,g)\cdot \alpha(\phi(g),f(h,g^{-1}))\cdot \eta(g,g^{-1})
\\&\eqcref{lem:funautnorm:a*}\alpha(\phi(g),f(h,g^{-1}))\cdot \eta(g,g^{-1})
\\&\eqcref{lem:funautnorm:d*} \alpha(\phi(g),\eta(t(h),g^{-1}))^{-1} \cdot \alpha(\phi(g), f(h,1))\cdot \eta(g,g^{-1})
\\&\eqcref{CA5*} \alpha(\phi(g),\eta(t(h)g^{-1},g))\cdot \alpha(\phi(g), f'(h))
\end{align*}
In the last step, we have applied \cref{CA5*} to the triple $(t(h),g^{-1},g)$; together with \cref{lem:funautnorm:c*} and the fact that $\eta$ is $U$-valued this gives
\begin{equation*}
\eta(t(h),g^{-1})\cdot \eta(t(h)g^{-1},g) = \eta(g^{-1},g)\text{.}
\end{equation*}
We have then applied \cref{CA5*} to the triple $(g,g^{-1},g)$, which gives
\begin{equation*}
\eta(g,g^{-1}) = \alpha(\phi(g),\eta(g^{-1},g))\text{.}
\end{equation*}
Together with the centrality of $U$ in $H$, this proves the last step. 

\item
is \cref{CA4*}.

\end{enumerate}
It remains to prove that $(\phi,f',\eta)$ is mapped to the functor given by $(\phi,f,\eta)$. We have to parse through the construction of \cref{sec:CI}, where $(\phi,f',\eta)$ defines the functor $F: \Gamma \to \Gamma$  by
\begin{equation*}
F(h,g):= (\eta(t(h),g)^{-1}\cdot f'(h),\phi(g))\text{.}
\end{equation*}
This functor's assignments are on objects $\phi$, and on morphisms
\begin{equation*}
(h , g) \mapsto \eta(t(h),g)^{-1}\cdot f'(h) \eqcref{lem:funautnorm:d*} f(h,g)\text{.}
\end{equation*}
Furthermore, the natural transformation defined by $\eta$ has exactly the component map $\eta$. 
\end{proof}

\begin{proposition}
\label{lem:aut:step3}
Under assumptions \cref{aut:1*,aut:2*}, the functor $\AUT_{CI}(\Gamma) \to \AUT^{fun}(\Gamma)$ is full and faithful. 
\end{proposition}

\begin{proof}
Consider two crossed intertwiners $(\phi,f,\eta)$ and $(\phi',f',\eta')$ and the associated UZ-normalized objects $(\phi,\tilde f,\eta)$ and $(\phi',\tilde f',\eta')$ of $\AUT^{fun}(\Gamma)$, where $\tilde f(h,g) := \eta(t(h),g)^{-1}\cdot f(h)$. We show the following. For a smooth map $\beta: G \to H$ the following conditions are equivalent:
\begin{enumerate}[(a)]

\item 
$\beta$ satisfies \cref{CA6*,CA7*,CA8*}.

\item
We have $\phi=\phi'$, $\mathrm{im}(\beta) \subset U$ and $\beta$ satisfies \cref{CT1*,CT2*}.
\end{enumerate}
(a) means that $\beta$ is a morphism in $\AUT^{fun}(\Gamma)$ between $(\phi,\tilde f,\eta)$ and $(\phi',\tilde f',\eta')$, and (b) means that $\beta$  is a morphism in $\CI (\Gamma,\Gamma)$ between $(\phi,f,\eta)$ and $(\phi',f',\eta')$. Thus,  equivalence of (a) and (b) shows the claim of the proposition.  

Suppose (a) holds. \cref{CA8*} implies by setting $g_1=g_2=1$ that $\beta(1)\eta(1,1) = \eta'(1,1)$; since our objects are normalized, we get $\beta(1)=1$. 
Using \cref{aut:1*,aut:2*} and arguing similar as in the proof of \cref{lem:th:aut:3:a}, we have that  $t \circ \beta:G \times G \to Z$ is constant, and due to $\beta(1)=1$ constantly $1$. This means that $\beta$ takes values in $U$. Now \cref{CA6*} implies $\phi=\phi'$. \cref{CA7*} implies
$\beta(t(h))\tilde f(h,1)=\tilde f'(h,1)$; this is \cref{CT1*}. \cref{CA8*} is \cref{CT2*}. This shows (b).

Conversely, suppose (b) holds. Then, we get \cref{CA6*} since $t(\beta(g))=1$. We get \cref{CA7*}:
\begin{align*}
\beta(t(h)g)\tilde f(h,g) 
&=\beta(t(h)g)\eta(t(h),g)^{-1} f(h)
\\&\eqcref{CT2*} \eta'(t(h),g)^{-1}\beta(t(h))\alpha(\phi(t(h)),\beta(g))f(h)
\\&\hspace{-0.7em}\eqcref{CI1*} \eta'(t(h),g)^{-1}\alpha(t(f(h)),\beta(g))\beta(t(h))f(h)
\\&= \eta'(t(h),g)^{-1}\beta(t(h))f(h)\beta(g)
\\&\eqcref{CT1*} \eta'(t(h),g)^{-1}f'(h)\beta(g)
\\&=\tilde f'(h,g)  \beta(g)\text{.}
\end{align*}
Here we have used that $\beta$ is $U$-valued and $H$ acts trivially on $U$. Finally, \cref{CA8*} is the same as \cref{CT2*}. This proves (a).
\end{proof}

\begin{remark}
Yet smaller versions of the automorphism 2-group of a strict 2-group $\Gamma$ can be obtained by requiring automorphisms to be strictly multiplicative. This can be done on the level of smooth functors and on the level of crossed homomorphisms, yielding a commutative diagram 
\begin{equation*}
\alxydim{}{\AUT_{CI}(\Gamma) \ar[r] & \AUT^{fun}(\Gamma) \\ \AUT_{SI}(\Gamma) \ar@{^(->}[u] \ar[r] & \AUT^{fun}_{str}(\Gamma)\text{,}\ar@{^(->}[u]}
\end{equation*} 
with the strictly multiplicative versions in the bottom line. Here, the objects of $\AUT_{SI}(\Gamma)$ are invertible strict intertwiners, i.e. crossed intertwiners $(\phi,f,\eta)$ with $\eta=1$, and the objects of $\AUT^{fun}_{str}(\Gamma)$ are invertible smooth functors that are strictly monoidal.   In general, even under the assumptions of \cref{th:auto}, the vertical inclusion functors in the above diagram are \emph{not} equivalences; in particular, the strict versions are  insufficient to describe  automorphisms of T-duality.

\end{remark}

\begin{remark}
\label{inner-automorphisms}
Each Lie 2-group $\Gamma$ furnishes a strict 2-group homomorphism
\begin{equation*}
i:\Gamma \to \AUT^{fun}_{str}(\Gamma)\text{,}
\end{equation*}
whose image consists -- by definition -- 
of the inner automorphisms \cite{roberts1}. If $\Gamma$ is given as a crossed module, then to an object $x \in \Gamma_0=G$ the functor $i$ assigns a smooth functor $i_x : \Gamma \to \Gamma$ given by $i_x(g):= xgx^{-1}$ and $i_x(h,g) := (\alpha(x,h),xgx^{-1})$. 
One can check that $i_x$ is a strictly monoidal, strictly invertible smooth functor.
To a morphism $(k,x) \in \Gamma_1=H \ltimes_{\alpha} G$ the functor $i$ assigns the smooth natural transformation $\lambda_{k,x}: i_x \Rightarrow i_{t(k)x}$ with $\lambda_{k,x}(g) := (k\alpha(xgx^{-1},k)^{-1},xgx^{-1})$.
While the functor $i_x$ can be described as a strict intertwiner ($i_x=(\phi_x,f_x,1)$ with $\phi_x(g):=xgx^{-1}$ and $f_x(h):=\alpha(x,h)$), the natural transformation $\lambda_{k,x}$ comes -- in general -- not from a crossed transformation, since these only exist between two crossed intertwiners with the \emph{same} $\phi$. In particular, the functor $i$ in general does not land in  $\AUT_{SI}(\Gamma)$. However, if  the Lie group $\Gamma_0$ is abelian, then $\phi_x$ is the identity on the level of objects, and  $\lambda_{k,x}$ does come from a crossed transformation, namely  $\beta_{k,x}(g) :=k\alpha(xgx^{-1},k)^{-1}$. 
This shows that, for $\Gamma_0$ abelian, the  functor $i$ factors (canonically) through the inclusion $\AUT_{SI}(\Gamma) \to \AUT^{fun}_{str}(\Gamma)$, and results into a 2-group homomorphism
\begin{equation*}
i:\Gamma \to\AUT_{SI}(\Gamma)\text{.}
\end{equation*} 
\end{remark}

\setsecnumdepth{2}

\section{A model for the automorphisms of T-duality}

\label{sec:auttdual}

We briefly introduce and discuss the split pseudo-orthogonal group $\mathrm{O}^{\pm}(n,n,\Z)$ in \cref{sec:onnZ}. 
In \cref{sec:modelautos} we construct explicitly a strict Lie 2-group $\AU[n]$. In \cref{sec:realizationauto}  we establish an equivalence of 2-groups $\AU[n] \cong \AUT(\TD[n])$ using \cref{th:auto}.

\subsection{The integral split pseudo-orthogonal group}

\label{sec:onnZ}

\begin{definition}
\label{eq:defskewiso}
The \emph{integral split pseudo-orthogonal group} $\mathrm{O}^{\pm}(n,n,\Z)$ is the group of isometries up to sign of the indefinite symmetric bilinear form
\begin{equation*}
I := \begin{pmatrix}0 & E_n \\
E_n & 0 \\
\end{pmatrix} \in \Z^{2n \times 2n}\text{,}
\end{equation*}
i.e. it consists of matrices $A\in \mathrm{GL}(2n,\Z)$ such that 
\begin{equation*}
A^{tr} \cdot I \cdot A= \pm  I\text{. }
\end{equation*}
\end{definition}

There are two interesting group homomorphisms:
\begin{equation*}
\det :\mathrm{O}^{\pm}(n,n,\Z) \to \Z_2
\quand
\mathrm{iso}: \mathrm{O}^{\pm}(n,n,\Z) \to \Z_2
\end{equation*}
where $\mathrm{iso}(A)$ indicates if $A$ is an isometry or a pseudo-isometry. Since $I^2=E_{2n}$ we have $A^{tr}IAI=\mathrm{iso}(A)E_{2n}$ for all $A\in \mathrm{O}^{\pm}(n,n,\Z)$; this gives a formula that can be used to compute $\mathrm{iso}(A)$ from a given matrix. 
We note that $\mathrm{iso}(A^{tr})=\mathrm{iso}(A)$ and $\mathrm{O}(n,n,\Z)=\mathrm{ker}(\mathrm{iso})$ is the proper \emph{integral split orthogonal group}, so that 
$\mathrm{O}^{\pm}(n,n,\Z)=\mathrm{O}(n,n,\Z) \ltimes \Z_2$.
\begin{remark}
\label{subgroup}
We consider the following  elements and subgroups of $\mathrm{O}(n,n,\Z)$, see \cite[§2.4]{Giveon1994}: 
\begin{enumerate}[(a)]

\item 
\label{subgroup-Z}
$I, -I\in \mathrm{O}(n,n,\Z)$.
We will later look at the subgroup $Z:=\left \langle I  \right \rangle\cong \Z/2\Z$.

\item
Let $V_{ij}\in \mathrm{GL}(2n,\Z)$ be the permutation matrix, i.e.,  $a\mapsto V_{ij}a$ permutes the $i$th and $j$th component of $a\in \Z^{2n}$. Then, $V_{ij}\in \mathrm{O}(n,n,\Z)$ if $j=i+n$, and we write $V_i := V_{i,i+n}$. 
We let $V:=\left \langle V_{i} \sep 1\leq i \leq n  \right \rangle \subset \mathrm{O}(n,n,\Z)$ be the subgroup generated by these permutations. Note that $V\cong (\Z/2\Z)^{n}$. The elements $V_i$ are called \quot{factorized duality} in \cite{Giveon1994}.   

\item
\label{subgroup-GLnZ}
$\mathrm{O}(n,n,\Z)$ receives an injective group homomorphism
\begin{equation*}
D:\mathrm{GL}(n,\Z) \to \mathrm{O}(n,n,\Z):A \mapsto D_A := \begin{pmatrix}
A & 0 \\
0 & (A^{tr})^{-1} \\
\end{pmatrix}\text{.}
\end{equation*}
In particular, this shows that $\mathrm{O}(n,n,\Z)$ is infinite and non-abelian for $n>1$.
The subgroups $Z\subset \mathrm{O}(n,n,\Z)$ and $\mathrm{O}(n,\Z) \subset \mathrm{GL}(n,\Z)\subset \mathrm{O}(n,n,\Z)$ commute:
$ID_A =D_AI$
for all $A\in \mathrm{O}(n,\Z)$. 
The elements $D_A$ are called \quot{base change} in \cite{Giveon1994}.

\item
\label{subgroup-sonZ}
Let $\mathfrak{so}(n,\Z)\subset \Z^{n \times n}$ be the additive group of skew-symmetric matrices with integer entries. There is an injective group homomorphism
\begin{equation*}
\mathfrak{so}(n) \to \mathrm{O}(n,n,\Z):  B \mapsto 
e^{B} :=\begin{pmatrix}
E_n & 0 \\
B & E_n \\
\end{pmatrix} \text{.}
\end{equation*}
The elements $e^{B}$ are called \quot{integer theta-parameter shift} in \cite{Giveon1994}.

\end{enumerate}
\end{remark}

\begin{remark}
\label{subgroup-K}
In addition to the subgroups of $\mathrm{O}(n,n,\Z)$ listed in \cref{subgroup} there is a Klein 4-group $K=\Z/2\Z \times \Z/2\Z$ that embeds into $\mathrm{O}^{\pm}(n,n,\Z)$ via
\begin{equation*}
(\alpha,\hat\alpha) \mapsto K_{\alpha,\hat\alpha}:= \begin{pmatrix}
\alpha E_n & 0 \\
0 & \hat\alpha E_n \\
\end{pmatrix}\text{.}
\end{equation*}
Note that while the subgroups (a) -- (d) are subgroups of $\mathrm{O}(n,n,\Z)$,  $K$ is a subgroup only of the larger group $\mathrm{O}^{\pm}(n,n,\Z)$, as $\mathrm{iso}(K_{\alpha,\hat\alpha})=\alpha\hat\alpha$. The subgroup $K$ was found in \cite{Kim2022} and was added here in a revised version.  
\end{remark}

\begin{example}
\label{ex:n=1}
For $n=1$, a matrix
\begin{equation*}
A=\begin{pmatrix}a & b \\
c &d \\
\end{pmatrix}\in \mathrm{GL}(2,\Z)
\end{equation*}
lies in $\mathrm{O}^{\pm}(n,n,\Z)$ if and only if
$ac=bd=0$ and $ad+bc=\pm 1$.
This nails it down to eight possible matrices:
\begin{center}
\begin{tabular}{cccc}
$\begin{pmatrix}1 & 0 \\
0 & 1 \\
\end{pmatrix}$ & $\begin{pmatrix}-1 & 0 \\
0 & -1 \\
\end{pmatrix}$ & $\begin{pmatrix}0 & 1 \\
1 & 0 \\
\end{pmatrix}$ & $\begin{pmatrix}0 & -1 \\
-1 & 0 \\
\end{pmatrix}$ \\
$\begin{pmatrix}0 & -1 \\
1 & 0 \\
\end{pmatrix}$ & $\begin{pmatrix}0 & 1 \\
-1 & 0 \\
\end{pmatrix}$ & $\begin{pmatrix}1 & 0 \\
0 & -1 \\
\end{pmatrix}$ & $\begin{pmatrix}-1 & 0 \\
0 & 1 \\
\end{pmatrix}$ 
\end{tabular}
\end{center}
We have $\mathrm{iso}(A)=ad+bc$. Hence, the first line contains the four matrices in $\mathrm{O}(n,n,\Z)$. Obviously, $\mathrm{O}(1,1,\Z)$ is a Klein four-group, and $\mathrm{O}^{\pm}(1,1,\Z)$ is the dihedral group $D_4$, where \begin{equation*}
R:=\begin{pmatrix}0 & -1 \\
1 & 0 \\
\end{pmatrix}
\quand
I=\begin{pmatrix}0 & 1 \\
1 & 0 \\
\end{pmatrix}
\end{equation*}
correspond to the rotation by an angle of $\frac{\pi}{2}$ and a reflection, respectively. 
\end{example}

\subsection{A 2-group version of the split pseudo-orthogonal group}

\label{sec:modelautos}

We define a strict 2-group $\AU[n]$ of which we will see in \cref{sec:realizationauto} that it models the automorphism 2-group of $\TD[n]$.
We consider pairs $(A,\eta)$ consisting of a matrix $A\in \mathrm{O}^{\pm}(n,n,\Z)$ and a smooth map 
\begin{equation*}
\eta: \R^{2n} \times \R^{2n} \to \ueins
\end{equation*}
satisfying the following conditions for all $a,a',a''\in \R^{2n}$ and $m,m'\in \Z^{2n}$:
\begin{enumerate}[({A}1),leftmargin=*]

\item
\label{AU:1}
$\eta(m,m')=0$

\item
\label{AU:2}
$\eta(m,a) +\mathrm{iso}(A) [a,m]=\eta(a,m)+[Aa,Am]$

\item
\label{AU:3}
$\eta(a,a')+ \eta(a+a',a'')=\eta(a',a'')+ \eta(a,a'+a'')$

\end{enumerate}
Here, $[-,-]: \R^{2n} \times \R^{2n} \to \R$ was introduced in \cref{sec:intro}, where we also agreed on the convention that $\ueins=\R/\Z$. 
We note that \cref{AU:3} is the usual cocycle condition for group 2-cocycles.
The set $\AU[n,0]$ of all such pairs is a group under the product
\begin{equation*}
(A_1,\eta_1)\cdot (A_2,\eta_2) := (A_1A_2,\tilde\eta)
\end{equation*}
with
\begin{align*}
\tilde\eta(a,a') &:=\eta_1(A_2a,A_2a')+\mathrm{iso}(A_1)  \eta_2(a,a')\text{.}
\end{align*}
The unit is $(E_{2n},0)$, and the inverse of $(A,\eta)$ is $(A^{-1},\eta^{-1})$ with
$\eta^{-1}(a,a'):=-  \mathrm{iso}(A) \eta(A^{-1}a,A^{-1}a')$.

\begin{remark}
\cref{AU:1*,AU:3*} imply $\eta(0,a)=\eta(a,0)=0$ for all $a\in \R^{2n}$.
\end{remark}

We define a 2-group $\AU[n]$ whose group of objects is $\AU[n,0]$. A morphism between $(A,\eta)$ and $(A',\eta')$ will exists only if $A=A'$. In this case, a morphism is a smooth map  $\beta: \R^{2n} \to \ueins$ satisfying
\begin{enumerate}[({A}1),leftmargin=*]
\setcounter{enumi}{3}

\item
\label{AU:4}
$\beta(m)=0$ for all $m\in \Z^{2n}$. 

\item
\label{AU:5}
$\beta(a_1)+\beta(a_2)+\eta(a_1,a_2) = \eta'(a_1,a_2)+  \beta
(a_1+a_2)$ for all $a_1,a_2\in \R^{2n}$. 

\end{enumerate}
Composition of morphisms is the point-wise addition, and the identity morphism is the zero map. The product (horizontal composition) of morphisms
\begin{equation*}
\beta_1:(A_1,\eta_1) \to (A_1',\eta'_1)
\quand
\beta_2:(A_2,\eta_2) \to (A_2',\eta'_2)
 \end{equation*}
 is given by
\begin{equation*}
(\beta_1 \cdot \beta_2)(a) := \beta_1(A_2a)+ \mathrm{iso}(A_1)\beta_2(a)\text{.} 
\end{equation*}
It is straightforward to show that this defines a strict 2-group $\AU [n]$.

\begin{remark}
\label{re:automorphismsofautomorphisms}
If $\beta:\R^{2n} \to \ueins$ is an \emph{auto}morphism of some object $(A,\eta)$, then it is a group homomorphism by \cref{AU:5*}, and by \ref{AU:4} even a group homomorphism $\T^{2n} \to \ueins$. Since the composition of automorphisms is their point-wise product, we have a group isomorphism
\begin{equation*}
\mathrm{Aut}_{\AU[n]}(A,\eta) =\mathrm{Hom}^{\infty}(\T^{2n},\ueins)\cong \Z^{2n}\text{.} \end{equation*}
\end{remark}

\subsection{Implementation as crossed automorphisms of T-duality}

\label{sec:realizationauto}

We define a strict 2-group homomorphism
\begin{equation*}
\mathbb{I}:\AU[n] \to \AUT_{CI}(\TD[n])
\end{equation*}
that implements the previously defined 2-group $\AU[n]$ as crossed intertwiners of $\TD[n]$. Let us recall the definition of the 2-group $\TD_n$, which was already given in \cref{sec:intro}.

\begin{definition}
\label{def-TDn}
The \emph{T-duality 2-group} is the crossed module $\TD_n := (G,H,t,\alpha)$ consisting of the Lie groups $G:=\R^{2n}$ and $H:=\Z^{2n} \times \ueins$, the Lie group homomorphism $t: H \to G$  given by  $t(m,s):=m$, and the action $\alpha$ of $G$ on $H$ given by  $\alpha(a,(m,s))=(m,s-[a,m])$. 
\end{definition}

Here,  $[-,-]:\R^{2n} \times \R^{2n} \to \R$ is the bilinear form with matrix 
\begin{equation*}
J=\begin{pmatrix}
0 & 0 \\
E_n & 0 \\
\end{pmatrix}\text{.}
\end{equation*}
The two axioms of the crossed module $\TD_n$ are trivially satisfied.
The central subgroup $U \subset H$ is $U=\ueins$, and the induced action of $G$ on $U$ is trivial. 

An object $(\phi,f,\eta)$ in $\AUT_{CI}(\TD_n)$ consists  of Lie group isomorphisms  
$\phi:\R^{2n} \to \R^{2n}$  and $f: \Z^{2n} \times \ueins \to \Z^{2n} \times \ueins$, and of a smooth map  $\eta: \R^{2n} \times \R^{2n} \to \ueins$, satisfying the axioms  \cref{CI1*,CI2*,CI4*,CI5*}. 
We notice that \cref{CI5*} simplifies a bit to
\begin{enumerate}[({CI}4-TD),leftmargin=*]

\item
\label{CI4-TD}
$\eta(a,a')+ \eta(a+a',a'')=\eta(a',a'')+ \eta(a,a'+a'')$,

\end{enumerate}
as the action $\alpha$ of $G$ on $U$ is trivial for $\TD_n$. In turn, this can be used to show  
\begin{align*}
\eta(a,m-a)&=\eta(a,-a+m) =- \eta(-a,m)+ \eta(a,-a)
\\
\eta(m-a,a)&=-\eta(m,-a)+ \eta(-a,a)
\\
\eta(a,-a)&=- \eta(0,a)+\eta(-a,a)+ \eta(a,0)=\eta(-a,a)
\text{.}
\end{align*}
Using these, \cref{CI4*} also simplifies to
\begin{enumerate}[({CI}3-TD),leftmargin=*]

\item
\label{CI3-TD}

$\eta(m,-a)+ f(\alpha(a,(m,s)))= \eta(-a,m)+ \alpha(\phi(a), f(m,s))$.
\end{enumerate}

 In order to construct the 2-group homomorphism $\mathbb{I}$, we associate to a pair $(A,\eta)$ the crossed intertwiner $\mathbb{I}(A,\eta):=(\phi_A,f_A,\eta)$ where $\phi_A(a) := Aa$ is matrix multiplication,   $f_A(m,s):= (Am,\mathrm{iso}(A)s)$, and $\eta$ is  what it is. 

\begin{lemma}
\label{ci-of-pairs}
The triple $(\phi_A,f_A,\eta)$ is a crossed intertwiner, and the assignment $(A,\eta) \mapsto (\phi_A,f_A,\eta)$ is an injective group homomorphism. 
\end{lemma}

\begin{proof}
$\phi_A$ and $f_A$ are clearly group homomorphisms, and injectivity is obvious. We check the axioms of a crossed intertwiner:
\begin{enumerate}[({CI}1),leftmargin=*]

\item
$\phi_A(t(m,s))=\phi_A(m)=Am=t(Am,\mathrm{iso}(A)s)=t(f_A(m,s))$.

\item
$\eta(t(m,s),t(m',s'))=\eta(m,m') \eqcref{AU:1*} 1$.

\item[(CI3-TD)]
is proved by the following calculation:
\begin{align*}
\eta(m,-a)+f_A(\alpha(a,(m,s)))
&=\eta(m,-a)+ f_A(m,s-[a,m])
\\&= \eta(m,-a)+ (Am,\mathrm{iso}(A)s-\mathrm{iso}(A)[a,m])
\\&\hspace{-0.55em}\eqcref{AU:2*} \eta(-a,m)+(Am,\mathrm{iso}(A)s-[Aa,Am])
\\&= \eta(-a,m)+ \alpha(Aa,(Am,\mathrm{iso}(A)s))
\\&=\eta(-a,m)+\alpha(\phi_A(a), f_A(m,s))
\end{align*}

\item[(CI4-TD)]
is exactly \cref{AU:3*}.

\end{enumerate}
We consider a product $(A_2,\eta_2)\cdot (A_1,\eta_1) = (A_2A_1,\tilde\eta)$. To $(A_2A_1,\tilde\eta)$ we assign $(\phi_{A_2A_1},f_{A_2A_1},\tilde\eta)$. On the other hand, we have from the composition of crossed intertwiners
\begin{align*}
(\phi_{A_2},f_{A_2},\eta_2) \circ (\phi_{A_1},f_{A_1},\eta_1) &=   (\phi_{A_2} \circ \phi_{A_1}, f_{A_2} \circ f_{A_1},\eta_2 \circ (\phi_{A_1} \times \phi_{A_1}) + f_{A_2} \circ \eta_1 )
\\&=(\phi_{A_2A_1},f_{A_2A_1},\tilde \eta)\text{.}
\end{align*}
This shows that we have a group homomorphism.
\end{proof}

\begin{lemma}
\label{lem:aut:equiv:morph}
For a smooth map $\beta: \R^{2n} \to \ueins$, the following are equivalent:
\begin{enumerate}[(a)]

\item 
it is a morphism  $(A,\eta)\to(A,\eta')$ in $\AU[n]$. 

\item
it is a crossed transformation $\beta:(\phi_A,f_A,\eta) \to (\phi_A,f_A,\eta')$.

\end{enumerate}
Moreover, vertical and horizontal composition in $\AU[n]$ and $\AUT_{CI}(\TD[n])$ coincide. 
\end{lemma}

\begin{proof}
Conditions for (a) are
\begin{enumerate}[({A}1),leftmargin=*]
\setcounter{enumi}{3}

\item
$\beta(n)=0$ for all $n\in \Z^{2n}$. 

\item
$\beta(a_1)+\beta(a_2)+\eta(a_1,a_2) = \eta'(a_1,a_2)+  \beta
(a_1+a_2)$ for all $a_1,a_2\in \R^{2n}$. 
\end{enumerate}
Conditions for (b) are:
\begin{enumerate}[({CT}1),leftmargin=*]

\item
$\beta(n)+ f_A(n,s)=f_A(n,s) $

\item
$\beta(a_1)+\beta(a_2)+\eta(a_1,a_2) = \eta'(a_1,a_2)+  \beta
(a_1+a_2)$. 

\end{enumerate}
These conditions are equivalent. The vertical composition is  in both cases addition. Horizontal composition in $\AU[n]$ is given by
\begin{equation*}
(\beta_1 \cdot \beta_2)(a) := \beta_1(A_2a)+ \mathrm{iso}(A_1)\beta_2(a) 
\end{equation*}
and horizontal composition in $\AUT_{CI}(\TD[n])$ is given by
\begin{equation*}
(\beta_1 \circ \beta_2)(a) := \beta_1(\phi_{A_2}(a))\cdot f_{A_1}(\beta_2(a))\text{.} 
\end{equation*}
These formulas coincide, too.
\end{proof}

Thus, we have constructed a 2-group homomorphism
\begin{equation*}
\mathbb{I}:\AU[n] \to \AUT_{CI}(\TD[n])\text{.}
\end{equation*}

\begin{theorem}
\label{prop:modelCI}
The 2-group homomorphism $\mathbb{I}$ is an equivalence, 
\begin{equation*}
\AU[n] \cong \AUT_{CI}(\TD[n])\text{.}
\end{equation*}
\end{theorem}

\begin{proof}
$\mathbb{I}$ is fully faithful by \cref{lem:aut:equiv:morph}. It remains to prove that our functor is essentially surjective.
Let $(\phi,f,\eta)$ be an object in $\AUT_{CI}(\TD[n])$. The first step is to change $(\phi,f,\eta)$ to an isomorphic object of $\AUT_{CI}(\TD[n])$.  We notice that
\cref{CI1*} shows that $f(m,s) = (\phi(m),\tilde f(m,s))$, for a Lie group homomorphism \begin{equation*}
\tilde f: \Z^{2n} \times \ueins \to \ueins\text{,}
\end{equation*}
and it also shows that $\phi(\Z^{2n}) \subset \Z^{2n}$, i.e. $\phi\in \mathrm{GL}(\Z^{2n})$.
We split $\tilde f$ further into two Lie group homomorphisms $\beta_{\Z}: \Z^{2n} \to \ueins$ and $f'': \ueins \to \ueins$ such that $\tilde f(m,s)=f''(s)-\beta_{\Z}(m)$.
Since $\Z^{2n} \subset \R^{2n}$ is discrete and $\ueins$ is connected, $\beta_{\Z}$ can be extended to a smooth map $\beta:\R^{2n} \to \ueins$ with $\beta|_{\Z^{2n}}=\beta_{\Z}$. 
We regard $\beta$ as a crossed transformation $\beta:(\phi,f,\eta) \Rightarrow (\phi,f',\eta')$, where $f'$ and $\eta'$ are defined such that \cref{CT1*,CT2*} are satisfied, i.e.,
\begin{align*}
f'(m,s) &:= \beta(m)+ f(m,s) 
\\
\eta'(a_1,a_2) &:= \beta(a_1)+\beta(a_2)+\eta(a_1,a_2)-\beta(a_1+a_2) \text{.}
\end{align*} 
By construction, $f'(m,s)=(\phi(m),f''(s))$. Summarizing, we can assume that our automorphism $(\phi,f,\eta)$ satisfies of $\phi(a)=Aa$ for a matrix $A\in \mathrm{GL}(2n,\Z)$ and $f(m,s)=(Am,\varepsilon s)$ for  $\varepsilon=\pm 1$ and $\eta$ is such that
\begin{enumerate}[(a)]

\item
$\eta(m,m') =0$ for all $m,m'\in \Z^{2n}$. 

\item
$\eta(m,a)+\varepsilon[a,m]=\eta(a,m)+[Aa,Am]$.

\item
$\eta(a,a')+ \eta(a+a',a'')=\eta(a',a'')+ \eta(a,a'+a'')\text{.}$
\end{enumerate}
The second step is to show that $A\in \mathrm{O}^{\pm}(n,n,\Z)$. 
Using the matrix $J$ of the bilinear form $[-,-]$, we  write (b) as:
\begin{equation*}
\eta(m,a)-\eta(a,m)= a^{tr} (A^{tr}J A-\varepsilon J)m\text{,}
\end{equation*}
which is -- after all -- an equation of $\ueins=\R/\Z$ numbers.
The left hand side is skew-symmetric under exchanging $a$ with $m$, and we claim that this implies that the matrix $A':=A^{tr}J A-\varepsilon J\in \Z^{n\times n}$ is skew-symmetric, too. Indeed, suppose $A'\in \Z^{n \times n}$ is a matrix such that for all $m\in \Z^{n}$ and $a\in \R^{n}$ we have $a^{tr}A' m=-m^{tr}A'a$ modulo $\Z$. That is, $a^{tr}(A'+A'^{tr})m\in \Z$, and so, for every $m\in \Z^{n}$, $d_m:=(A'+A'^{tr})m\in \Z^{n}$ has an integer-valued pairing against all $a\in \R^{n}$. This implies $d_m=0$. Since this is true for every $m\in \Z^{n}$, we conclude $A'+A'^{tr}=0$, proving the claim.

We obtain then\begin{equation*}
0=A'+A'^{tr}=-\varepsilon J +A^{tr}J A-\varepsilon J^{tr}+A^{tr} J^{tr}A=-(\varepsilon I-A^{tr}IA)\text{.}
\end{equation*}
This shows that $A\in \mathrm{O}^{\pm}(n,n,\Z)$, and $\varepsilon=\mathrm{iso}(A)$. 
The last step is to see that $(A,\eta)$ is an object of $\AU[n]$ -- indeed, (a), (b), and (c) above yield \cref{AU:1*,AU:2*,AU:3*}. Moreover, it is clear that $\mathbb{I}(A,\eta)=(\phi,f,\eta)$. 
\end{proof}

\begin{remark}
Up to 2-isomorphism, $\TD_n$ has no non-trivial inner automorphisms. This can be seen in two different ways. First, suppose the crossed intertwiner $\mathbb{I}(A,\eta)=(\phi_A,f_A,\eta)$ defined in \cref{ci-of-pairs} is an inner automorphism. Then, $\phi_A=\id_{\R^{2n}}$  (since $\R^{2n}$ is abelian) and $\eta=1$ (since  inner automorphisms are always strict, see \cref{inner-automorphisms}).
This implies  $(A,\eta)=(E_{2n},1)$, the neutral element. Since by \cref{prop:modelCI} \emph{all} automorphisms are in the essential image of $\mathbb{I}$, this shows the claim. The second argument is by a direct computation. If $x\in \R^{2n}$ is an object of $\TD_n$, then the associated inner automorphism  $i_x:\TD_n \to \TD_n$ is the identity on the level of objects, and given by $i_x((m,s),a)=(\alpha(x,(m,s)),a)=(s-[x,m],a)$ on the level of morphisms,  see  \cref{inner-automorphisms}. Now, one can check that the formula $\eta_x(a) := ((0,[x,a]),a)$ defines a natural isomorphism $\eta_x: i_x \Rightarrow \id_{\TD_n}$. 
\end{remark}

\section{Classification of the automorphism 2-group of T-duality}

\label{sec:classification}

In this section we perform computations with the 2-group $\AU[n]$ of which we know by \cref{th:auto,prop:modelCI} that it is the full automorphism 2-group $\AUT(\TD[n])$.

\subsection{Review of the classification of 2-groups}

We review the classification theory for 2-groups following Baez-Lauda \cite{baez5}, only rephrasing it in the language of  extensions of 2-groups.

For a strict 2-group $\Gamma$ we consider the group $\pi_0\Gamma$ of isomorphism classes of objects and the abelian group $\pi_1\Gamma:=\mathrm{Aut}(1) \subset \Gamma_1$. We recall that $\alpha_{\Gamma}([g], \gamma) := \id_{g} \cdot \gamma \cdot \id_{g^{-1}}$ is a well-defined action by group homomorphisms of $\pi_0\Gamma$ on $\pi_1\Gamma$.  
We regard the inclusion $i:\pi_1\Gamma \to \Gamma_1$ and the projection $p:\Gamma_0 \to \pi\Gamma_{dis}$ as 2-group homomorphisms (i.e., monoidal functors)
\begin{equation*}
B\pi_1\Gamma \to \Gamma \to \pi_0\Gamma_{dis}\text{.}
\end{equation*} 
This sequence of 2-group homomorphisms is exact in the sense of Vitale \cite{Vitale2002}, which coincides with the setting of Schommer-Pries \cite{pries2} when the latter is restricted from smooth to discrete 2-groups.

Conversely, if $G$ is a group and $A$ is an abelian Lie group, then a sequence
\begin{equation*}
BA \to \Gamma \to G_{dis}
\end{equation*}
of 2-group homomorphisms is  exact, if and only if $\Gamma_0 \to G$ is the canonical projection $\Gamma_0 \to \pi_0\Gamma$ followed by an isomorphism $\pi_0\Gamma\cong G$ and  $i$ induces an isomorphism $A \cong \pi_1\Gamma$. The formula $\alpha(g , a) := \id_{\tilde g} \cdot a \cdot \id_{\tilde g^{-1}}$, where $\tilde g\in \Gamma_0$ is any object lifting $g$, defines again an action by group homomorphisms of $G$  on $A$, which is intertwined under the isomorphisms $\pi_0\Gamma\cong G$ and $A\cong \pi_1\Gamma$ with the canonical action $\alpha_{\Gamma}$.
We recall that an extension of 2-groups is called \emph{central} when the induced action is trivial; however, in this paper we are concerned with non-central extensions.

\begin{theorem}[{\cite{baez5}}]
\label{th:class}
Let $G$ be a group, $A$ be an abelian group, and let $\alpha:G \times A \to A$ be an action of $G$ on $A$ by group homomorphisms. Then, there is a canonical bijection
\begin{equation*}
\bigset{11em}{Equivalence classes of  strict 2-group extensions\\ $BA \to \Gamma \to G_{dis}$\\ with induced action $\alpha$ } \cong \h^3(G,A^{\alpha})\text{.}
\end{equation*}
\end{theorem}

Here, $A^{\alpha}$ denotes $A$ considered as a $G$-module under the action $\alpha$, and $\h^3(G,A^{\alpha})$ denotes group cohomology. 
The cohomology class corresponding to a 2-group $\Gamma$ under \cref{th:class} is called the \emph{k-invariant} of $\Gamma$.

The bijection of \cref{th:class} is obtained in the following way, see \cite[§8.3]{baez5}: in the larger bicategory of \emph{coherent 2-groups}, every \emph{strict} 2-group $\Gamma$ is isomorphic to a  \emph{special} 2-group $\Gamma_s$, and special 2-groups are easily be seen to be  classified up to isomorphism by their k-invariant in  $\h^3(G,A^{\alpha})$. We remark for completeness that \cref{th:class} holds verbatim for general \emph{coherent} 2-groups instead of just \emph{strict} ones, since every coherent 2-group is isomorphic (as a central extension) to a strict one.

We also remark that the bijection of \cref{th:class} is a bijection of pointed sets. A strict 2-group extension with k-invariant zero is equivalent  to the semi-direct product $BA \ltimes_{\alpha} G$ as defined in \cite[Def. A.8]{Nikolause}; explicitly, the underlying Lie-groupoid of $BA \ltimes_{\alpha} G$ is $BA \times G_{dis}$, and the multiplication (on the level of morphisms) is the one of the semi-direct product $A \ltimes_{\alpha} G$.

The following lemma describes a  well-known method to compute the k-invariant of a strict 2-group, for which I was, however, not able to find a reference.  
\begin{lemma}
\label{lem:class}
Suppose $BA \to \Gamma \to G_{dis}$ is a strict 2-group extension with  induced action $\alpha$. Let  $S: G \to \Gamma_0$ be a section against the projection $\Gamma_0 \to G$ such that $S(1)=1$. Then, for each pair of elements $A,B \in G$ there exists an isomorphism $\beta_{A,B}: S(A)\cdot S(B) \to S(AB)$ such that $\beta_{1,A}$ and $\beta_{A,1}$ are identities. In turn, there exists a unique 3-cocycle $\xi : G^3 \to A^{\alpha}$ such that the diagram   
\begin{equation*}
\alxydim{@C=-1em@R=3.5em}{&& S(A) \cdot S(B) \cdot S(C) \ar[drr]^-{\id_{S(A)} \cdot \beta_{B,C}} \\  S(AB) \cdot S(C)  \ar@{<-}[urr]^-{\beta_{A,B} \cdot \id_{S(C)}} \ar[dr]_{\beta_{AB,C}} &&&& S(A) \cdot S(BC) \ar[dl]^{\beta_{A,BC}} \\ & S(ABC) \ar[rr]_-{\xi_{A,B,C} \cdot \id_{S(ABC)}} &\hspace{9em}& S(ABC) }
\end{equation*}
is commutative for all $A,B,C\in G$. Then, the class of $\xi$ in $\h^3(G,A^{\alpha})$ is independent of the choices of $S$ and $\beta_{A,B}$, and it is the k-invariant of $\Gamma$.
Moreover, if $\xi =0$, then $S: G_{dis} \to \Gamma$ is a weak 2-group homomorphism, and the composite $BA \ltimes_{\alpha} G \stackrel{i \times S}\to \Gamma \times\Gamma\stackrel{m}{\to} \Gamma$ is an equivalence of 2-groups. 
\end{lemma}

\begin{proof}
It is easy to deduce from the fact that $\Gamma$ is strict and that $S(1)=1$ that  $\beta_{A,B}$ as claimed exist.
It is then clear that an automorphism $b_{A,B,C}$ of $S(ABC)$ exists uniquely such that the diagram is commutative with $b_{A,B,C}$ at the bottom arrow.
Pasting five copies of the above diagram, we obtain an identity for $b_{A,B,C}$ when confronted with four group elements $A,B,C,D\in G$. This identity is
\begin{multline*}
(\beta_{ABC,D} \circ (b_{A,B,C} \cdot \id_{S(D)}) \circ \beta_{ABC,D} ^{\circ-1})\circ b_{A,BC,D} \circ (\beta_{A,BCD}^{\circ-1}\circ (\id_{S(A)}\cdot b_{B,C,D}) \circ \beta_{A,BCD})\\=b_{A,B,CD}\circ b_{AB,C,D}\text{.}
\end{multline*}
Here, all five main factors are automorphisms of $S(ABCD)$. Note that the automorphism group of any object is abelian (under composition), and that
$(\alpha \cdot \id_A) \circ (\beta\cdot \id_A)=(\alpha\circ\beta)\cdot \id_A$ holds for arbitrary composable morphisms $\alpha$, $\beta$, and objects $A$. Using these rules, and defining
\begin{equation*}
\xi_{A,B,C} := b_{A,B,C} \cdot \id_{S(ABC)^{-1}} \in \pi_1\Gamma\text{,}
\end{equation*}
we obtain
 \begin{multline*}
(\beta_{ABC,D} \circ (\xi_{A,B,C} \cdot \id_{S(ABC)S(D)}) \circ \beta_{ABC,D} ^{\circ-1})
\circ (\beta_{A,BCD}^{\circ-1}\circ (\id_{S(A)}\cdot \xi_{B,C,D}\cdot \id_{S(A)^{-1}}\cdot \id_{S(A)S(BCD)}) \circ \beta_{A,BCD})\\=(\xi_{A,B,CD}\circ\xi_{A,BC,D}^{-1} \circ \xi_{AB,C,D})\cdot \id_{S(ABCD)}\text{.}
\end{multline*}
In order to treat the first two factors we use the identity
\begin{equation*}
\alpha\circ (\xi \cdot \id_X)\circ\alpha^{-1}=\xi\text{,}
\end{equation*} 
which holds for all isomorphisms $\alpha:X \to 1$ and all $\xi\in \pi_1\Gamma$. It follows from the exchange law in a straightforward way. 
From this get
\begin{equation*}
\beta_{ABC,D} \circ (\xi_{A,B,C} \cdot \id_{S(ABC)S(D)}) \circ \beta_{ABC,D} ^{\circ-1}=\xi_{A,B,C}  \cdot \id_{S(ABCD)}
\end{equation*}
and
\begin{multline*}
(\beta_{A,BCD}^{\circ-1}\circ (\id_{S(A)}\cdot \xi_{B,C,D}\cdot \id_{S(A)^{-1}}\cdot \id_{S(A)S(BCD)}) \circ \beta_{A,BCD})=(\id_{S(A)}\cdot \xi_{B,C,D}\cdot \id_{S(A)^{-1}})\cdot \id_{S(ABCD)}
\end{multline*}
Summarizing, we end up with
\begin{equation*}
\xi_{A,B,C}  \cdot(\id_{S(A)}\cdot \xi_{B,C,D}\cdot \id_{S(A)^{-1}})= \xi_{A,B,CD}\circ\xi_{A,BC,D}^{-1} \circ \xi_{AB,C,D}\text{.} 
\end{equation*}
This is precisely the cocycle condition for $\xi$,
\begin{equation*}
\alpha_{\Gamma}(A, \xi_{B,C,D})-\xi _{AB,C,D}+\xi_{A,BC,D}-\xi_{A,B,CD}+\xi_{A,B,C}=0\text{.}
\end{equation*}
It is straightforward to see from the identities $\beta_{A,1}=\id$ and $\beta_{1,A}=\id$ that $\xi$ is normalized.
Now let $\Gamma_s$ be the special 2-group corresponding to $\xi$, i.e., $\Gamma_s$ is the skeletal  monoidal category with objects $G$ and morphisms $A \ltimes_{\alpha} G$,  trivial unitors and associator
$a_{A,B,C} := (\xi_{A,B,C},ABC)$.
It remains to show that $\Gamma\cong \Gamma_s$ as coherent 2-groups; this proves that $\xi$ represents the k-invariant of $\Gamma$. In order to see this, we construct a functor $F: \Gamma_s \to \Gamma$. On objects, it is given by $S$. On morphisms, we send $(\xi,A)$ to $\xi \cdot \id_{S(A)}$; this yields a functor.
Next we claim that our isomorphisms $\beta_{A,B}$ upgrade $F$ to a monoidal functor. Recall that $F(1)=1$ and that  $\beta_{A,1}$ and $\beta_{1,A}$ are identities. This means that we only have to show a single identity: 
\begin{align*}
F(a_{A,B,C}) \circ \beta_{AB,C} \circ (\beta_{A,B} \cdot \id_{S(C)}) &= \beta_{A,BC} \circ (\id_{S(A)} \cdot \beta_{B,C}) \text{.}
\end{align*}
But this identity follows immediately from the definition of $F$, since 
\begin{equation*}
F(a_{A,B,C})=F(\xi_{A,B,C} ,ABC)=\xi_{A,B,C} \cdot \id_{S(ABC)}=b_{A,B,C}\text{.}
\end{equation*}
Finally, it is easy to see that $F$ is essentially surjective and fully faithful, and hence an equivalence of categories. Thus, it is also an equivalence of monoidal categories, and thus an equivalence between $\Gamma_s$ and $\Gamma$ as coherent 2-groups. If $\xi=0$, then $\Gamma_s=BA \ltimes_{\alpha} G$, and $F$ factors as $BA \ltimes_{\alpha} G_{dis} \stackrel{i \times S}{\to} \Gamma \times \Gamma \stackrel{m}\to \Gamma$, which hence is an equivalence. 
\end{proof}

\subsection{Computation of the homotopy   groups}

\label{sec:homtopycrossedmodule}

In this section we determine the homotopy groups $\pi_1\Gamma$ and $\pi_0\Gamma$ of the strict 2-group $\AUT(\TD[n])$, and the induced action of $\pi_0\Gamma$ on $\pi_1\Gamma$. We perform all computations with the 2-group $\AU[n]$ which is isomorphic to $\AUT(\TD[n])$ by \cref{prop:modelCI,th:auto}.

\cref{re:automorphismsofautomorphisms} implies immediately:

\begin{proposition}
\label{prop:pi1}
$\pi_1(\AU[n]) \cong \Z^{2n}\text{.}$
\end{proposition}

\noindent
The computation of $\pi_0(\AU[n])$ is more difficult.

\begin{proposition}
\label{prop:pi0}
The projection 
\begin{equation*}
P:\pi_0(\AU[n]) \to \mathrm{O}^{\pm}(n,n,\Z):(A,\eta) \mapsto A
\end{equation*}
is an isomorphism of groups.
\end{proposition}

It is clear that the map in \cref{prop:pi0} is well-defined and a group homomorphism. We prove in the following that it is injective, and show in \cref{lem:section} below that it is surjective.

\begin{proof}[Proof of injectivity in \cref{prop:pi0}]
We show injectivity by considering an object  in the kernel, i.e., one of the form $(0,\eta)$, and constructing an isomorphism $(0,\eta)\cong(0,0)$  in $\AU[n]$.  Such object $(0,\eta)$ is a smooth map $\eta: \R^{2n} \times \R^{2n} \to \ueins$ satisfying  \cref{AU:1,AU:3} as well as 
\begin{enumerate}[({A}2'),leftmargin=*]

\item
\label{AU:2'}
$\eta(m,a) =\eta(a,m)$ for all $a\in \R^{2n}$ and $m\in \Z^{2n}$.

\end{enumerate}
We want to relate $\eta$ to a central extension of $\T^{2n}$ by $\ueins$ as Lie groups. Such extensions, 
\begin{equation*}
1 \to \ueins \to A \to \T^{2n} \to 1\text{,}
\end{equation*}
are classified by the smooth group cohomology $\h^2_{sg}(\T^{2n},\ueins)$ of Brylinski \cite{brylinski3} and Segal-Mitchison \cite{segal4}. A cocycle in $\h^2_{sg}(\T^{2n},\ueins)$ w.r.t. to the covering $\R^{2n} \to \T^{2n}$ consists of a pair $(g,\gamma)$ in which $g:\R^{2n} \times_{\T^{2n}} \R^{2n} \to \ueins$ is a \v Cech 2-cocycle, and $
\gamma: \R^{2n} \times \R^{2n} \to \ueins$ is a group 2-cocycle, and the condition
\begin{equation}
\label{eq:compggamma}
g(a_1,a_1')+ g(a_2,a_2')+ \gamma(a_1',a_2') = \gamma(a_1,a_2)+ g(a_1+a_2,a_1'+a_2')
\end{equation}
is satisfied for all $((a_1,a_2),(a_1',a_2')) \in (\R^{2n} \times \R^{2n}) \times_{\T^{2n} \times \T^{2n}} (\R^{2n} \times \R^{2n})$.   

We shall translate fibre products of $\R^{2n} \to \T^{2n}$ under the diffeomorphism $\R^{2n} \times_{\T^{2n}} \R^{2n} \cong \R^{2n} \times \Z^{2n}$ given by $(a,a')\mapsto (a,a'-a)$.
Then, we consider instead of the \v Cech 2-cocycle $g$ a smooth map $\tilde g: \R^{2n} \times \Z^{2n} \to \ueins$ satisfying
\begin{equation}
\label{eq:eqcech}
\tilde g(a,m)+\tilde g(a+m,m')=\tilde g(a,m+m')\text{,}
\end{equation}  
and instead of  \cref{eq:compggamma} we obtain the equation
\begin{equation}
\label{eq:compggammaZ}
\tilde g(a_1,m_1)+ \tilde g(a_2,m_2)+ \gamma(a_1+m_1,a_2+m_2) = \gamma(a_1,a_2)+ \tilde g(a_1+a_2,m_1+m_2)
\end{equation}
for all $a_1,a_2\in \R^{2n}$ and $m_1,m_2\in \Z^{2n}$. 
Returning to our map $\eta:\R^{2n} \times \R^{2n} \to \ueins$, we set
\begin{equation}
\label{def-extension}
\gamma := \eta
\quand
\tilde g := \eta|_{\R^{2n} \times \Z^{2n}}\text{.}
\end{equation}
Then, $\gamma$ is a group 2-cocycle due to \cref{AU:3}, and $\tilde g$ satisfies \cref{eq:eqcech}:
\begin{multline*}
\tilde g(a,m+m') =\eta(a,m+m')
\eqcref{AU:3}\eta(a,m)+\eta(a+m,m')-\eta(m,m')
\\\eqcref{AU:1} \eta(a,m)+\eta(a+m,m')
= \tilde g (a,m) + \tilde g(a+m,m)\text{.} 
\end{multline*}
Finally, condition \cref{eq:compggammaZ} is  satisfied:
\begin{align*}
&\mquad\tilde g(a_1,m_1)+ \tilde g(a_2,m_2)+ \gamma(a_1+m_1,a_2+m_2)
\\&=\eta(a_1,m_1)+ \eta(a_2,m_2)+ \eta(m_1+a_1,a_2+m_2)
\\&\hspace{-0.5em}\eqcref{AU:3} \eta(a_1,m_1)+ \eta(a_2,m_2)-\eta(m_1,a_1)+\eta(m_1,a_1+a_2+m_2)+\eta(a_1,a_2+m_2)
\\&\hspace{-0.5em}\eqcref{AU:2'}  \eta(a_2,m_2)+\eta(m_1,m_2+a_1+a_2)+\eta(a_1,a_2+m_2)
\\&\hspace{-0.5em}\eqcref{AU:3}  \eta(a_2,m_2)+\eta(m_1,m_2)+\eta(m_1+m_2,a_1+a_2)-\eta(m_2,a_1+a_2)+\eta(a_1,a_2+m_2)
\\&\hspace{-0.5em}\eqcref{AU:1}\eta(a_2,m_2)-\eta(m_2,a_1+a_2)+\eta(a_1,a_2+m_2)+\eta(m_1+m_2,a_1+a_2)
\\&\hspace{-0.5em}\eqcref{AU:2'}\eta(a_2,m_2)-\eta(a_1+a_2,m_2)+\eta(a_1,a_2+m_2)+\eta(a_1+a_2,m_1+m_2)
\\&\hspace{-0.5em}\eqcref{AU:3} \eta(a_1,a_2)+\eta(a_1+a_1,m_1+m_2)
\\&=\gamma(a_1,a_2)+ \tilde g(a_1+a_2,m_1+m_2)\text{.}
\end{align*}
We conclude that $\eta$ defines, via \cref{def-extension}, a central Lie group extension of $\T^{2n}$ by $\ueins$. 
Now we infer that all such extensions are trivial: this comes from the fact that  $\h^2_{sg}(\T^{2n},\ueins)\cong \h^3(B\T^{2n},\Z)$ by a result of Schommer-Pries   \cite{pries2},  whereas the cohomology of $B\T^{2n}$ is free on even generators. Thus, any cocycle $(\gamma,\tilde g)$ trivializes, which means that there exists a smooth map $\beta: \R^{2n} \to \ueins$ such that 
\begin{equation}
\label{trivializing-chain}
g(a,a')+\beta(a)=\beta(a')
\quand
\gamma(a_1,a_2)=\beta(a_1+a_2)-\beta(a_1)-\beta(a_2)
\end{equation}
hold for all $a,a',a_1,a_2\in \R^{2n}$. In case  of our cocycle given by  \cref{def-extension}, we show that any trivializing map $\beta$ establishes a morphism $(0,\eta) \cong (0,0)$ in $\AU[n]$. Indeed, \cref{trivializing-chain} implies via  \cref{AU:1} that 
\begin{equation*}
\beta(m)=\beta(m')\quand \beta(m_1+m_2)=\beta(m_1)+\beta(m_2)
\end{equation*}
hold for all $m,m'\in \Z^{2n}$,
i.e., the restriction of $\beta$ to $\Z^{2n}$ is a constant group homomorphism. Thus, $\beta|_{\Z^{2n}}=0$; this shows \cref{AU:4}.  The second equation in \cref{trivializing-chain} provides \cref{AU:5}.
\end{proof}

To proceed, we define a section
\begin{equation*}
S: \mathrm{O}^{\pm}(n,n,\Z) \to \mathbb{A}_{n,0}
\end{equation*}
against the projection of \cref{prop:pi0}.
 We consider  the matrix 
\begin{equation*}
J=\begin{pmatrix}
0 & 0 \\
E_n & 0 \\
\end{pmatrix}
\end{equation*}
from the definition of the 2-group $\TD[n]$, whose corresponding bilinear form is  $[-,-]$, and define, for  $A\in \mathrm{O}^{\pm}(n,n,\Z)$, the matrix
\begin{equation*}
\B A := \mathrm{iso}(A) J - A^{tr}J A \in \Z^{2n \times 2n}\text{.}
\end{equation*}
Note that $\B A$ is the matrix of the bilinear form
\begin{equation}
\label{eq:BA}
\braket{a}{\B A}{b} =\mathrm{iso}(A)[a,b]-[Aa,Ab]\text{.}
\end{equation}
It is easy to see that $\B A$ is skew-symmetric.
If $B \in \Z^{2n \times 2n}$ is any skew-symmetric matrix, then we denote by $B_{low} \in \Z^{2n \times 2n}$ the lower triangular matrix with zeroes on the diagonal, satisfying $B=B_{low}-B_{low}^{tr}$. We will then write
$\lbraket - B -$ for the bilinear form corresponding to the matrix $B_{low}$.
We define a bilinear map $\eta_A: \R^{2n} \times \R^{2n} \to \ueins$ by
\begin{equation}
\label{def:etaA}
\eta_A(a,a') :=  \lbraket a{\B A}{a'}\text{.}
\end{equation}

\begin{lemma}
\label{lem:section}
For any $A\in \mathrm{O}^{\pm}(n,n,\Z)$, the pair $(A,\eta_A)$ is an object in $\AU[n]$. In particular, the map 
\begin{equation*}
S: \mathrm{O}^{\pm}(n,n,\Z) \to \AU[n]: A \mapsto (A,\eta_A)
\end{equation*}
is a section against the projection $P$, and $P$ is surjective. 
\end{lemma}

\begin{proof}
Every bilinear form automatically satisfies the cocycle condition \cref{AU:3*}, and \cref{AU:1*} is also clear from the definition of $\eta_A$. It remains to show \ref{AU:2}:
\begin{multline*}
\eta_A(m,a) -\eta_A(a,m) 
=  \lbraket{m}{\B A}{a}-\lbraket{a}{\B A}{m}
\\= \braket{m}{\B A}{a}
=- \braket{a}{\B A}{m}=-\mathrm{iso}(A) [a,m]+[Aa,Am]\text{.}\qedhere
\end{multline*}
\end{proof}

\begin{example}
We describe $\eta_A$ for particular $A\in \mathrm{O}^{\pm}(n,n,\Z)$.
\begin{enumerate}[(a)]

\item
\label{ex:split:a}
For $A=E_{2n}$ we get $\B {E_{2n}} = 0$ and thus $\eta_{E_{2n}}=0$.

\item 
\label{ex:split:b}
For $n=1$ and 
\begin{equation*}
A = 
\begin{pmatrix}
a & b \\
c & d \\
\end{pmatrix}
\end{equation*}
we have (see \cref{ex:n=1})
\begin{equation*}
\B A = \begin{pmatrix}
0 & -cb \\
bc & 0 \\
\end{pmatrix}\text{.}
\end{equation*}
Thus, $\eta_A(x\oplus \hat x,x'\oplus \hat x') = bc\hat xx'$. 

\item
\label{ex:split:c}
For $I\in \mathrm{O}(n,n,\Z)$ we have 
\begin{equation*}
\B I = J - IJ I =J-J^{tr} = \begin{pmatrix}0 & -E_n \\
E_n & 0 \\
\end{pmatrix}\text{.}
\end{equation*}
Hence, $\B I_{low}=J$ and thus $\eta_I(a,a')=\lbraket{a}{\B I}{a'}=\braket {a}{J}{a'}=[a,a']$. 

\item
For $V_{i} \in \mathrm{O}(n,n,\Z)$ we have
\begin{equation*}
\B {V_{i}}=  J - V_{i,i+n}J V_{i,i+n}  = E_{i+n,i}-E_{i,i+n}\text{.}
\end{equation*}
where $E_{ij}$ is the $(2n\times 2n)$-matrix with a $1$ at the $(i,j)$-position and zeros otherwise. Hence, $\B {V_{i}}_{low}=E_{i+n,i}$ and $\eta_{V_{i}}(a,a') =a_{i+n}a'_i$. 

\item
\label{ex:split:d}
We consider $A \in \mathrm{GL}(n,\Z)$ with corresponding
\begin{equation*}
D_A=\begin{pmatrix}
A & 0 \\
0 & (A^{tr})^{-1} \\
\end{pmatrix}\in \mathrm{O}(n,n,\Z)\text{,}
\end{equation*}
 see \cref{sec:onnZ}. We get $\B {D_A}=0$  and hence $\eta_{D_A} = 0$.

\item
\label{ex:split:e}
Consider $B\in \mathfrak{so}(n,\Z)$ and the corresponding element $e^{B} \in \mathrm{O}(n,n,\Z)$. We obtain
\begin{equation*}
\B {e^{B}} =\begin{pmatrix}
B & 0 \\
0 & 0 \\
\end{pmatrix} 
\end{equation*}
and hence $\eta_{e^{B}}(a \oplus b,a' \oplus b') = \lbraket{a}{B}{a'}$. 

\item
\label{ex:split:g}
Consider the elements $K_{\alpha,\hat\alpha}$ of the Klein 4-group $K\subset \mathrm{O}^{\pm}(n,n,\Z)$. We obtain $\B {K_{\alpha,\hat\alpha}}=0$
and hence $\eta_{K_{\alpha,\hat\alpha}}=0$. 

\end{enumerate}
\end{example}

Now that we have computed the homotopy groups of $\AU[n]$, it remains to compute the action. 

\begin{lemma}
\label{lem:action}
The action of $A\in \mathrm{O}^{\pm}(n,n,\Z)$ on $m\in \Z^{2n}$ is given by 
\begin{equation*}
(A,m) \mapsto IAI m  \text{.}
\end{equation*} 
\end{lemma}

\begin{proof}
Let $\beta$ be an automorphism of $(E_{2n},0)$, i.e. $\beta:\T^{2n} \to \ueins$ is a Lie group homomorphism. The action of an object $(A,\eta)$ on $\beta$ is given by
\begin{align*}
((A,\eta),\beta)\mapsto& \id_{(A,\eta)} \cdot \beta \cdot \id_{(A,\eta)^{-1}}
=  \mathrm{iso}(A)\beta \cdot \id_{(A^{-1},\eta^{-1})}
=  \mathrm{iso}(A) (\beta\circ A^{-1}) 
\text{.}
\end{align*}
We recall that the isomorphism $\Z^{2n} \cong \mathrm{Hom}^{\infty}(\T^{n},\ueins)$ is given by $m \mapsto \beta_m$ with $\beta_m:\T^{2n} \to \ueins$ defined by $\beta_m(a)=ma$ (standard scalar product mod $\Z$). Note that \cref{eq:defskewiso}  implies
$\mathrm{iso}(A)A^{-1} = IA^{tr}I$. 
From this we get
\begin{multline*}
\mathrm{iso}(A)(\beta_m \circ A^{-1}) (a)=\beta_m(\mathrm{iso}(A)A^{-1}a)=\beta_m(IA^{tr}Ia)=m^{tr}IA^{tr}Ia=(IAIm)^{tr} a= \beta_{IAIm}(a)\text{.}
\end{multline*}
This shows the claim.
\end{proof}

We may summarize the results of this section as follows.

\begin{theorem}
\label{th:extension}
The 2-group  $\AU[n]$ is a non-central extension
\begin{equation*}
1 \to \Z^{2n} \to \AU[n] \to \mathrm{O}^{\pm}(n,n,\Z)_{dis} \to 1\text{,} 
\end{equation*}
whose induced action of $\mathrm{O}^{\pm}(n,n,\Z)$ on $\Z^{2n}$ is $(A,m) \mapsto IAIm$.
\end{theorem}

\begin{remark}
We note that $A \mapsto IAI$ is an (inner) automorphism of $\mathrm{O}^{\pm}(n,n,\Z)$. One could thus change the projection $p:\mathbb{A}^{\pm}_{n} \to \mathrm{O}^{\pm}(n,n,\Z)$ by this automorphism, and thereby achieve  that the action in \cref{lem:action} is plain matrix multiplication. 
\end{remark}

\subsection{Computation of the k-invariant}

\label{sec:splitting:multiplicativity}

We will use the map $S: \mathrm{O}^{\pm}(n,n,\Z) \to \AU[n,0]$  from \cref{lem:section} to determine the k-invariant of $\AU[n]$ via \cref{lem:class}.
Note that $S$ satisfies $S(1)=1$ by \cref{ex:split:a}.  The next step is find isomorphisms
\begin{equation*}
\beta_{A,B}: S(A) \cdot S(B) \to S(AB)
\end{equation*} 
in $\AU[n]$, for any $A,B \in \mathrm{O}^{\pm}(n,n,\Z)$. 
We have $S(AB)=(AB,\eta_{AB})$ and write $S(A) \cdot S(B) =: (AB,\tilde\eta)$, with 
\begin{align*}
\tilde\eta(a,a') 
&=\eta_A(Ba,Ba')+\mathrm{iso}(A)\cdot  \eta_B(a,a')
\\&= \braket{a}{B^{tr}\B A_{low}B}{a'}+\mathrm{iso}(A)\cdot \braket{a}{\B B_{low}}{a'}
\\&= \braket{a}{\X AB}{a'}\text{,}
\end{align*}
where we have defined 
\begin{equation*}
\X AB :=B^{tr}\B A_{low}B +\mathrm{iso}(A) \B B_{low} \in \Z^{2n \times 2n}
\end{equation*} 
as the matrix of $\tilde\eta.$
In order to find $\beta_{A,B}$, we need to compare the bilinear forms $\eta_{AB}$ and $\tilde\eta$, and thus their matrices $\B{AB}_{low}$ and $\X AB$.

In general, $\X AB \neq \B {AB}_{low}$ and hence $\eta_{AB}\neq \tilde\eta$. In order to compare the matrices $\X AB$ and $\B {AB}_{low}$ we first claim that
\begin{equation}
\label{eq:cocycleB}
\B {AB} = B^{tr}\B AB+\mathrm{iso}(A)\B B\text{,}
\end{equation}
which can easily be checked from the definition of $\B A$.
From this we obtain that
\begin{equation*}
\B {AB}=\X AB-\X AB^{tr}\text{.}
\end{equation*}
Since $\B {AB}_{low}-\B {AB}_{low}^{tr}=\B {AB}=\X AB-\X AB^{tr}$, the difference \begin{equation}
\label{eq:HABdef}
H_{A,B} :=\X AB-\B {AB}_{low}\in \Z^{2n \times 2n}
\end{equation}
is a symmetric matrix. An explicit formula for this matrix is
\begin{equation*}
H_{A,B} = B^{tr}\B A_{low}B-(B^{tr}\B AB) _{low}\text{.}
\end{equation*} 

The following treatment of symmetric integral matrices was kindly explained to the author by Nora Ganter. 
Suppose $H\in \Z^{k \times k}$ is a symmetric matrix. Define
$\beta_H: \R^{k} \to \ueins$
by 
\begin{equation}
\label{eq:defbetaH}
\beta_H(x) := \frac{1}{2}x^{tr}Hx -\frac{1}{2} \sum_{i=1}^{k} H_{ii}x_i\text{.}
\end{equation}

\begin{lemma}
The map $\beta_H$ satisfies the following conditions:
\begin{enumerate}[(a)]

\item 
\label{lem:evilnumbertheory:a}
$\beta_H(x+y)-\beta_H(x)-\beta_H(y)=x^{tr}Hy$ for all $x,y\in \R^{k}$ 

\item
\label{lem:evilnumbertheory:b}
$\beta_H(x)=1$ for all $x\in \Z^{n}$. 

\end{enumerate}
\end{lemma}

\begin{proof}
Since the second summand of $\beta_H$ is linear in $x$, it drops out of the formula in (a). Thus we have, due to the symmetry of $H$,
\begin{equation*}
\beta_H(x+y)-\beta_H(x)-\beta_H(y)= \frac{1}{2}(x+y)^{tr}H(x+y)- \frac{1}{2}x^{tr}Hx- \frac{1}{2}y^{tr}Hy=x^{tr}Hy\text{.}
\end{equation*}
For (b), we we note that $\beta_H(x)\in \frac{1}{2}\Z$ for $x\in \Z^{k}$. On the other hand, we show now that  $\beta_H(x)$ is zero mod $\Z$. First of all, we have, again using symmetry,
\begin{equation*}
\frac{1}{2}x^{tr}Hx =\frac{1}{2} \sum_{i=1}^{k}H_{ii}x_i^2 + \sum_{1\leq i<j \leq n} H_{ij}x_ix_j=\frac{1}{2}\sum_{i=1}^{k}H_{ii}x_i^2 \mod \Z\text{.}
\end{equation*}
It now remains to show that
\begin{equation*}
\frac{1}{2}\sum_{i=1}^{k}H_{ii}(x_i^2-x_i)=0 \mod \Z\text{.}
\end{equation*}
But $x_i^2-x_i$ is always even, and the claim follows. 
\end{proof}

We apply this to $H_{A,B}$ and define $\beta_{A,B} := \beta_{H_{A,B}}: \R^{2n} \to \ueins$.

\begin{lemma}
$\beta_{A,B}$ is a morphism $\beta_{A,B}: S(A) \cdot S(B) \to S(AB)$ in $\AU[n]$. 
\end{lemma}

\begin{proof}
\cref{lem:evilnumbertheory:b} gives Axiom \ref{AU:4}. Axiom \ref{AU:5} is
\begin{equation*}
\beta_{A,B}(x,y)+\beta_{A,B}(x',y')+\eta_{AB}((x,y),(x',y')) = \tilde\eta((x,y),(x',y'))+  \beta_{A,B}
(x+x',y+y')\text{.}
\end{equation*}
By \cref{lem:evilnumbertheory:a}, this is  equivalent to the definition $H_{A,B}=\X AB-\B {AB}_{low}$. 
\end{proof}

\begin{example}
For $n=1$ we have
\begin{align*}
\B {AB}_{low} &=\begin{pmatrix}
0 & 0 \\
A_{11}B_{12}A_{22}B_{21}+A_{12}B_{22}A_{21}B_{11} & 0 \\
\end{pmatrix} 
\\
\X AB &= \begin{pmatrix}
0 & B_{12}B_{21}A_{12}A_{21} \\
B_{11}B_{22}A_{12}A_{21}+\mathrm{iso}(A)B_{12}B_{21} & 0 \\
\end{pmatrix}
\end{align*}
so that
\begin{equation*}
H_{A,B} =\begin{pmatrix}
0 & B_{12}B_{21}A_{12}A_{21} \\
B_{12}B_{21}A_{12}A_{21} & 0 \\
\end{pmatrix} \text{.}
\end{equation*}
Thus, the formula for $\beta_{A,B}$  becomes simpler because $H_{A,B}$ is zero on the diagonal (for higher $n$, this is in general not the case).
We obtain
\begin{equation}
\label{eq:n1beta}
\beta_{A,B}(x,y) = B_{12}B_{21}A_{12}A_{21}xy\text{.}
\end{equation}
\end{example}

\begin{example}
\label{ex:mult:flip}
For $A,B \in Z=\{E_{2n},I\}$ we have the following table:
\begin{center}
\begin{tabular}{c|c|c|c|c|c}
$A$ & $B$ & $\B {AB}_{low}$ & $\X AB$ & $H_{A,B}$ & $\beta_{A,B}$ \\\hline
$E_{2n}$ & $E_{2n}$ & $0$ & $0$ & $0$ & $0$ \\
$E_{2n}$ & $I$ & $J$ & $J$ & $0$ & $0$ \\
$I$ & $E_{2n}$ & $J$ & $J$ & $0$ & $0$ \\
$I$ & $I$ & $0$ & $I$ & $I$ & $\beta_{I,I}$  \\
\end{tabular}
\end{center}
Here, the only non-trivial thing is to compute the items in the bottom row:
\begin{align*}
\X AB &= I^{tr}\B I_{low}I +\mathrm{iso}(I) \B I_{low} =I J I+ J =I\text{.}
\end{align*}
Note that $\beta_{I,I}$ is an isomorphism $\beta_{I,I}: S(I)^2 \to S(I^2)=1$. As a map $\beta_{I,I}:\R^{2n} \to \ueins$ is the quadratic form associated to the matrix $I$:
\begin{align*}
\beta_{I,I}(x) &=\frac{1}{2}xIx^{tr} - \frac{1}{2}\sum_{i=1}^{2n}I_{ii}x_i =\frac{1}{2}xIx^{tr}\text{.} 
\end{align*}
Since $\mathrm{O}(n,n,\Z)$ is the isometry group of $I$, we see that $\beta_{I,I}$ is $\mathrm{O}(n,n,\Z)$-invariant, i.e. $\beta_{I,I}(Ax)=\beta_{I,I}(x)$ for all $x\in \R^{2n}$ and $A\in \mathrm{O}(n,n,\Z)$.
\end{example}

\begin{remark}
\label{ex:flipautomorphism}
The crossed intertwiner associated by \cref{prop:modelCI} to the object $S(I)=(I,\eta_I)$ discussed in \cref{ex:split:c} is the automorphism $flip$ of \cref{ex:flip}. 
The crossed transformation $flip^2 \to \id$ associated by \cref{prop:modelCI} to the isomorphism    $\beta_{I,I}: S(I)^2 \to 1$ coincides with the crossed transformation $\beta$ discussed in \cref{ex:flip}.
From \cref{ex:coherence:I} we obtain the desired coherence law for $\beta$.
\end{remark}

\begin{example}
\label{ex:mult:Vi}
For $A,B \in \{E_{2n},V_{i}\}$ we have the following table:
\begin{center}
\begin{tabular}{c|c|c|c|c|c}
$A$ & $B$ & $\B {AB}_{low}$ & $\X AB$ & $H_{A,B}$ & $\beta_{A,B}$ \\\hline
$E_{2n}$ & $E_{2n}$ & $0$ & $0$ & $0$ & $0$ \\
$E_{2n}$ & $V_i$ & $E_{i+n,i}$ & $E_{i+n,i}$ & $0$ & $0$ \\
$V_i$ & $E_{2n}$ & $E_{i+n,i}$ & $E_{i+n,i}$ & $0$ & $0$ \\
$V_i$ & $V_i$ & $0$ & $E_{i,n+i}+E_{i+n,i}$ & $-E_{i,n+i}-E_{i+n,i}$ & $\beta_{V_i,V_i}$  \\
\end{tabular}
\end{center}
Here, the only non-trivial thing is to compute the items in the bottom row:
\begin{align*}
\X AB &= V_i^{tr}\B {V_i}_{low}V_i +\mathrm{iso}(V_i) \B {V_i}_{low} =V_iE_{i+n,i}V_i + E_{i+n,i}=E_{i,n+i}+E_{i+n,i}
\\
\beta_{V_i,V_i}(x) &= x_{i}x_{i+n}
\end{align*}
\end{example}

\begin{example}
\label{ex:splitting:gl}
Under restriction to the $\mathrm{GL}(n,\Z)$ subgroup we have $\B {AB}_{low}=0$ and $\X AB =0$; hence $H_{A,B}=0$ and $\beta_{A,B}=1$. The same is true for the subgroup $K$.
\end{example}

\begin{example}
\label{ex:splitting:so}
Under restriction to the $\mathfrak{so}(n,\Z)$-subgroup we have
\begin{equation*}
(B_{e^{A}e^{B}})_{low} =(B_{e^{A+B}})_{low}= \begin{pmatrix}
A_{low} + B_{low} & 0 \\
0 & 0 \\
\end{pmatrix}=(e^{B})^{tr} \, (B_{e^{A}})_{low} \,e^{B} + (B_{e^{B}} )_{low}=\X AB\text{.} 
\end{equation*}
Thus, $H_{e^{A},e^{B}}=0$ and hence $\beta_{e^{A},e^{B}}=0$. 
\end{example}

\label{sec:splitting:coherence}

Now that we have discussed the splitting $S$ and the isomorphisms $\beta_{A,B}$ we are in position to extract the cocycle $\xi_{A,B,C}$ representing the k-invariant, still following \cref{lem:class}.
The formula is\begin{align*}
\xi_{A,B,C} &=(\beta_{A,BC} \circ (\id_{S(A)} \cdot \beta_{B,C}) \circ (\beta_{A,B} \cdot \id_{S(C)})^{\circ -1}\circ \beta_{AB,C}^{\circ-1})\cdot \id_{S(ABC)^{-1}}\text{.}
\end{align*}
Using  multiplication of isomorphisms in $\AU[n]$, we have:
\begin{align*}
(\beta_{A,B} \cdot \id_{S(C)})(x) = \beta_{A,B}(Cx) 
\quand
(\id_{S(A)}\cdot \beta_{B,C})(x) = \mathrm{iso}(A)\beta_{B,C}(x)
\end{align*}
Further, $S(ABC)^{-1}=(C^{-1}B^{-1}A^{-1},\eta_{ABC}^{-1})$. 
Thus,
\begin{multline}
\label{eq:error}
\xi_{A,B,C}(x)=\beta_{A,BC}(C^{-1}B^{-1}A^{-1}x)+\mathrm{iso}(A)\beta_{B,C}(C^{-1}B^{-1}A^{-1}x)\\-\beta_{A,B}(B^{-1}A^{-1}x)-\beta_{AB,C}(C^{-1}B^{-1}A^{-1}x)\text{.}
\end{multline}
This is, in the first place, a smooth map $\xi_{A,B,C}:\R^{2n} \to \ueins$. As proved in \cref{re:automorphismsofautomorphisms}, it induces a Lie group homomorphism $\T^{2n} \to \ueins$, which, under the isomorphism of \cref{prop:pi1} is identified with $m_{A,B,C}\in\Z^{2n}$. Under this identification,  $\xi_{A,B,C}(x) = m_{A,B,C}\cdot x$ for all $x\in \R^{2n}$. Note that $\xi_{A,B,C}$ is a 3-cocycle on $\pi_0(\AU[n])=\mathrm{O}^{\pm}(n,n,\Z)$ with values in $\pi_1(\AU[n])$. Under the isomorphism of \cref{prop:pi1}, $m_{A,B,C}$ is then a 3-cocycle with values in $\Z^{2n}$, with respect to the action of \cref{lem:action}.
We summarize this as the following result.
\begin{proposition}
The k-invariant of the extension
\begin{equation*}
1 \to \Z^{2n} \to\AU[n] \to \mathrm{O}^{\pm}(n,n,\Z)\to 1 
\end{equation*}
is the class in $\h^3(\mathrm{O}^{\pm}(n,n,\Z),\Z^{2n})$ represented  
 by the cocycle $m:\mathrm{O}^{\pm}(n,n,\Z)^3 \to \Z^{2n}$ with $m_{A,B,C}\cdot x=\xi_{A,B,C}(x)$ for all $x\in \R^{2n}$ and $\xi_{A,B,C}$ defined in \cref{eq:error}. 
\end{proposition}    

We will derive a more explicit formula for $m_{A,B,C}$ in \cref{sec:torsion}. Many properties of the k-invariant can, however, already be seen now. 

\begin{lemma}
\label{prop:n1}
The cocycle $m$ vanishes when $n=1$.
\end{lemma}

\begin{proof}
We obtain from the formula \cref{eq:n1beta}:
\begin{align*}
\beta_{AB,C}(x \oplus y) &=  -C_{12}C_{21}(A_{11}B_{12}A_{22}B_{21}+A_{12}B_{22}A_{21}B_{11})xy
\\
\mathrm{iso}(A)\beta_{B,C}(x \oplus y) &=-(A_{11}A_{22}+A_{12}A_{21})C_{12}C_{21}B_{12}B_{21}xy
\\
\beta_{A,BC}(x \oplus y) &= -(B_{11}C_{12}B_{22}C_{21}+B_{12}C_{22}B_{21}C_{11})A_{12}A_{21}xy
\\
\beta_{A,B}(C(x \oplus y)) &= -B_{12}B_{21}A_{12}A_{21}(C_{11}C_{22}+C_{12}C_{21})xy
\end{align*}
Substituting these terms in  \cref{eq:error} gives zero.
\end{proof}

\begin{lemma}
\label{lem:resglso}
The cocycle $m$ vanishes when  restricted to $\mathrm{GL}(n,\Z)$, $\mathfrak{so}(n,\Z)$, or $K$.
\end{lemma}

\begin{proof}
This follows from \cref{ex:splitting:gl,ex:splitting:so}.
\end{proof}

\begin{lemma}
\label{ex:coherence:I}
The cocycle $m$ vanishes when  restricted to  $Z \subset \mathrm{O}(n,n,\Z)$.
\end{lemma}

\begin{proof}
From \cref{ex:mult:flip} we obtain:
\begin{center}
\begin{tabular}{c|c|c|c|c|c|c}
$A$ & $B$ & $C$ & $\beta_{AB,C}$ & $\beta_{A,B}$ & $\beta_{A,BC}$ & $\beta_{B,C}$  \\\hline
$E_{2n}$ & $E_{2n}$ & $E_{2n}$ & $0$ & $0$ & $0$ & $0$  \\
$E_{2n}$ & $E_{2n}$ & $I$ & $0$ & $0$ & $0$ & $0$   \\
$E_{2n}$ & $I$ & $E_{2n}$ & $0$ & $0$ & $0$ & $0$  \\
$E_{2n}$ & $I$ & $I$ & $\beta_{I,I}$ & $0$ & $0$ & $\beta_{I,I}$ \\
$I$ & $E_{2n}$ & $E_{2n}$ & $0$ & $0$ & $0$ & $0$  \\
$I$ & $E_{2n}$ & $I$ & $\beta_{I,I}$ & $0$ & $\beta_{I,I}$ & $0$ \\
$I$ & $I$ & $E_{2n}$ & $0$ & $\beta_{I,I}$ & $\beta_{I,I}$ & $0$ \\
$I$ & $I$ & $I$ & $0$ & $\beta_{I,I}$ & $0$ & $\beta_{I,I}$ \\
\end{tabular}
\end{center}
We compute \cref{eq:error} for each row, using that  $\mathrm{iso}(I)=1$ and that $\beta_{I,I}$ is $\mathrm{O}(n,n,\Z)$-invariant. Hence, the first factor in \cref{eq:error} reduces to
\begin{equation*}
\beta_{A,BC}(x)+\beta_{B,C}(x)-\beta_{A,B}(x)-\beta_{AB,C}(x)
\end{equation*}
It is obvious that this gives zero in each row. \end{proof}

\begin{lemma}
\label{lem:cocV}
The cocycle $m$ vanishes when  restricted to $V\subset \mathrm{O}(n,n,\Z)$.
\end{lemma}

\begin{proof}
From \cref{ex:mult:Vi} we obtain the same table as in the proof of \cref{ex:coherence:I} just with $I$ replaced by $V_i$. Since $\beta_{V_i,V_i}$ is not $\mathrm{O}(n,n,\Z)$-invariant, the expression from \cref{eq:error} is here
\begin{align}
\label{eq:errorn}
\beta_{A,BC}(x)+\beta_{B,C}(x)-\beta_{A,B}(Cx)-\beta_{AB,C}(x)\text{.}
\end{align} 
For the  four rows with $C=0$ this is irrelevant, and the first three and the  fifth row is trivial anyway. Rows six is as before because $\beta_{A,B}=0$. Only in row eight we need to compute new:
\begin{equation*}
\beta_{V_i,V_i}(x)=x_{i}x_{i+n}=x_{i+n}x_i=(V_ix)_i(V_ix)_{i+n}=\beta_{V_i,V_i}(V_ix)\text{.}
\end{equation*}
Thus, \cref{eq:errorn} gives again zero. 
\end{proof}

It is interesting to note that in all of the above lemmas, the cocycle $m$ representing the k-invariant is identically equal to zero. In order to further explore this, we  consider the functor $\mathbb{I}_S$ defined as the composition
\begin{equation*}
\alxydim{}{\mathrm{O}^{\pm}(n,n,\Z)_{dis} \ar[r]^-{S} & \AU[n] \ar[r]^-{\mathbb{I}} & \AUT_{CI}(\TD[n]) \to \AUT(\TD[n])\text{.}}
\end{equation*}
By \cref{prop:pi0}, it induces a group isomorphism
\begin{equation*}
\pi_0\mathbb{I}_S: \mathrm{O}^{\pm}(n,n,\Z) \to \pi_0\AUT(\TD[n])\text{.}
\end{equation*}
As explained in  \cref{sec:intro}, this can be seen as an action by homotopy equivalences on the classifying space $|B\TD[n]|$, whose restriction to $\mathrm{O}(n,n,\Z) \subset \mathrm{O}^{\pm}(n,n,\Z)$ has been discussed in \cite{bunke2006a}.
Due to the non-vanishing of the cocycle $m$, the functor $\mathbb{I}_S$ itself is not monoidal, i.e., it is not a 2-group homomorphism: the action by homotopy equivalences cannot be improved to a homotopy-coherent action of $\mathrm{O}^{\pm}(n,n,\Z)$ on $\TD[n]$. However, above lemmas show that such improvement can be obtained in special cases.

\begin{proposition}
\label{prop:actions}
The functor 
\begin{equation*}
\mathbb{I}_S: \mathrm{O}^{\pm}(n,n,\Z)_{dis} \to \AUT(\TD[n])
\end{equation*}
is a 2-group homomorphism if either
\begin{enumerate}

\item 
$n=1$, or

\item
after restriction to one of the subgroups $\mathfrak{so}(n,\Z)$, $\mathrm{GL}(n,\Z)$, $Z$, $V$, or $K$.

\end{enumerate}
Thus, we obtain  homotopy-coherent actions of $\mathrm{O}^{\pm}(1,1,\Z)$ on $\TD_1$, and of the groups $\mathfrak{so}(n,\Z)$, $\mathrm{GL}(n,\Z)$, $Z$, $V$, and $K$ on $\TD[n]$.
\end{proposition}

\noindent
The geometric meaning of these  homotopy-coherent group actions will be further explored in \cref{sec:Tautomorphisms}.

The case that the cocycle $m$ is non-vanishing but represents $0\in \h^3( \mathrm{O}^{\pm}(n,n,\Z),\Z^{2n})$ may or may not occur for certain $n>1$ or after restriction to other subgroups of $\mathrm{O}^{\pm}(n,n,\Z)$, but I do not know any examples of this kind. In such cases, a trivializing coboundary could be used to change our section $S$ to another section $S'$ for which then $\mathbb{I}_{S'}$ is monoidal. As remarked in the introduction, it is not known if $\h^3( \mathrm{O}^{\pm}(n,n,\Z),\Z^{2n})=0$.   

\subsection{The k-invariant is 2-torsion}

\label{sec:torsion}

In this section we prove that the k-invariant of $\AUT(\TD[n])$ is 2-torsion in $\h^3(\mathrm{O}^{\pm}(n,n,\Z),\Z^{2n})$. 
We copy from \cref{eq:error}
\begin{multline*}
\xi_{A,B,C}(x)=\beta_{A,BC}(C^{-1}B^{-1}A^{-1}x)+\mathrm{iso}(A)\beta_{B,C}(C^{-1}B^{-1}A^{-1}x)\\-\beta_{A,B}(B^{-1}A^{-1}x)-\beta_{AB,C}(C^{-1}B^{-1}A^{-1}x)\text{.}
\end{multline*}
and recall from \cref{eq:defbetaH} that
\begin{equation*}
\beta_{A,B}(x)=\frac{1}{2}x^{tr}H_{A,B}x -\frac{1}{2}  H^{diag}_{A,B}  \cdot x\text{,}
\end{equation*}
where we have introduced the notation
\begin{equation*}
()^{diag}: \Z^{2n \times 2n} \to \Z^{2n}
\end{equation*}
that extracts the diagonal of a matrix as a vector.
Since we know that $\xi_{A,B,C}$ is a group homomorphism and hence linear, the quadratic terms  cancel, and we may write the result as
\begin{equation*}
\xi_{A,B,C}(x) =-\frac{1}{2}(  H^{diag}_{A,BC}+\mathrm{iso}(A)  H^{diag}_{B,C}- C^{tr} H^{diag}_{A,B}- H^{diag}_{AB,C})\cdot C^{-1}B^{-1}A^{-1}x\text{.}
\end{equation*}
Under the isomorphism $\mathrm{Hom}^{\infty}(\T^{2n},\ueins) \cong \Z^{2n}$ we hence have
\begin{equation*}
m_{A,B,C} =-\frac{1}{2}A^{tr-1}B^{tr-1}C^{tr-1}( H^{diag}_{A,BC}+\mathrm{iso}(A)  H^{diag}_{B,C}
-  C^{tr}H^{diag}_{A,B}- H^{diag}_{AB,C})\text{.}
\end{equation*}
We may now go further and use that
\begin{equation*}
H_{A,B}^{diag} \eqcref{eq:HABdef} -\X AB^{diag}= -(B^{tr}\B A_{low}B)^{diag}\text{.}
\end{equation*}
Substituting this, we obtain the following.

\begin{proposition}
\label{prop:coex}
The expression
\begin{multline*}
m_{A,B,C} =-\frac{1}{2}A^{tr-1}B^{tr-1}C^{tr-1}\left (C^{tr}(B^{tr}(\B A)_{low}B)^{diag}  +(C^{tr}(\B {AB})_{low}C)^{diag}\right .
\\\left .\qquad-\mathrm{iso}(A)(C^{tr}(\B B)_{low}C)^{diag}-(C^{tr}B^{tr}(\B A)_{low}BC)^{diag} \right) \end{multline*}
defines a 3-cocycle on the group $\mathrm{O}^{\pm}(n,n,\Z)$ with values in  $\Z^{2n}$ w.r.t. the action of \cref{lem:action}, and it represents the k-invariant of $\AUT(\TD[n])$.  
\end{proposition}

The explicit description of the cocycle $m$ given in \cref{prop:coex} allows us to prove the following result.

\begin{proposition}
\label{prop:torsion}
The k-invariant of $\AUT(\TD[n])$ is 2-torsion in $\h^3(\mathrm{O}^{\pm}(n,n,\Z),\Z^{2n})$.
\end{proposition}

\begin{proof}
We define the expression
\begin{equation*}
\gamma_{A,B}:=A^{tr-1}B^{tr-1}(B^{tr}(\B A)_{low}B)^{diag}\text{,}
\end{equation*}
which is a 2-cochain on $\mathrm{O}^{\pm}(n,n,\Z)$ with coefficients  in $\Z^{2n}$.
We claim that 
\begin{equation*}
(\delta\gamma)_{A,B,C}=2m_{A,B,C}\text{,}
\end{equation*}
which follows by a direct computation: 
\begin{align*}
(\delta\gamma)_{A,B,C}&=\alpha(A, \gamma_{B,C})-\gamma(AB,C)+\gamma(A,BC)-\gamma(A,B)
\\&=A^{tr-1}B^{tr-1}C^{tr-1}\big (\mathrm{iso}(A)(C^{tr}(\B B)_{low}C)^{diag}-(C^{tr}(\B {AB})_{low}C)^{diag}
\\&\qquad+(C^{tr}B^{tr}(\B A)_{low}BC)^{diag}-C^{tr}(B^{tr}(\B A)_{low}B)^{diag}\big )
\end{align*} 
Here, $\alpha$ denotes the action of $\mathrm{O}^{\pm}(n,n,\Z)$ on $\Z^{2n}$ which we computed in \cref{lem:action}, and which we use here in the form $\alpha(A,m) =\mathrm{iso}(A)A^{tr-1}m$ described there. The result coincides with 2 times the expression from \cref{prop:coex}. 
\end{proof}

The Bockstein sequence $\Z^{2n} \stackrel{\cdot 2}\to \Z^{2n} \to \Z^{2n}/2\Z^{2n}$ is a sequence of $\mathrm{O}^{\pm}(n,n,\Z)$-modules under the action  of \cref{lem:action}, and hence induces a long exact sequence in group cohomology.
By \cref{prop:torsion}, the k-invariant of $\AUT(\TD[n])$ lies in the image of the connecting homomorphism 
\begin{equation*}
\alxydim{}{
\h^2(\mathrm{O}^{\pm}(n,n,\Z),\Z^{2n}/2\Z^{2n}) \ar[r]^-{\delta} & \h^3(\mathrm{O}^{\pm}(n,n,\Z),\Z^{2n})\text{.} }
\end{equation*}
The construction of the connecting homomorphism in fact tells us that 
\begin{equation*}
\tilde\gamma_{A,B} :=\gamma_{A,B} \mod 2\Z^{2n}
\end{equation*}
is a cocycle whose class $[\tilde\gamma] \in \h^2(\mathrm{O}^{\pm}(n,n,\Z),\Z^{2n}/2\Z^{2n})$ is a preimage of the k-invariant under $\delta$. On the other hand, the 2-cocycle $\tilde\gamma$ defines an (ordinary) group extension
\begin{equation*}
1\to \Z^{2n}/2\Z^{2n} \to \widetilde{\mathrm{O}^{\pm}(n,n,\Z)} \to \mathrm{O}^{\pm}(n,n,\Z) \to 1\text{.} 
\end{equation*} 
The following result is an elementary consequence, but worthwhile to note.

\begin{proposition}
The pullback of the extension
\begin{equation*}
1 \to B\Z^{2n} \to \AUT(\TD[n]) \to \mathrm{O}^{\pm}(n,n,\Z)_{dis} \to 1
\end{equation*}
along the group homomorphism
\begin{equation*}
\widetilde{\mathrm{O}^{\pm}(n,n,\Z)} \to \mathrm{O}^{\pm}(n,n,\Z)
\end{equation*} 
splits, i.e., it has trivial k-invariant. 
\end{proposition}

The Bachelor's thesis of Berkenhagen \cite{Berkenhagen2021} produces a  canonical trivialization of the pullback of $\AUT(\TD[n])$ to $\widetilde{\mathrm{O}^{\pm}(n,n,\Z)}$. It equips the pullback of the section $S$ to $\widetilde{\mathrm{O}^{\pm}(n,n,\Z)}$ with modified isomorphisms $\tilde \beta_{A,B}$ turning it into a 2-group homomorphism
\begin{equation*}
S: \widetilde{\mathrm{O}^{\pm}(n,n,\Z)}_{dis} \to \AUT(\TD[n])\text{.}
\end{equation*}
Thus, the group $\widetilde{\mathrm{O}^{\pm}(n,n,\Z)}$ acts coherently on $\TD[n]$, instead of the group $\mathrm{O}^{\pm}(n,n,\Z)$ which only acts by homotopy equivalences. In order to further use this action, it would be good to have a geometric interpretation of the group $\widetilde{\mathrm{O}^{\pm}(n,n,\Z)}$, for example, as the automorphism group of some geometric object. For the moment, this has to be left to further research.   

\section{Automorphisms and T-duality correspondences}

\label{sec:Tautomorphisms}

As recalled in \cref{sec:intro}, the T-duality 2-group $\TD[n]$ classifies topological T-duality correspondences of $\T^{n}$-bundles. The equivalence of \cref{eq:eq} induces a bijection
\begin{equation*}
\h^1(X,\TD[n]) \cong \hc 0 (\tcor_n(X))
\end{equation*} 
between the  \quot{non-abelian} cohomology of $X$ with values in $\TD[n]$, and equivalence classes of topological T-duality correspondences of $\T^{n}$-bundles over $X$.
 In \cref{sec:class} we formulate this bijection in concrete terms, so that in \cref{sec:actions} we are in position   to identify how certain automorphisms of $\TD[n]$ act geometrically on T-duality correspondences.  

\subsection{Local data for T-duality correspondences}

\label{sec:class}

We consider cocycles in the cohomology $\h^1(X,\TD[n])$ of a smooth manifold $X$ with coefficients in $\TD[n]$. The theory behind this is basically the one of Giraud \cite{giraud} and is  explained in  modern context in \cite[§A.3]{Nikolause}. A cocycle with respect to an open cover $\{U_i\}_{i\in I}$ of $X$  is (see \cite[Rem. 3.7]{Nikolause}) a tuple $(a,\hat a,m,\hat m,t)$ consisting of numbers $m_{ijk},\hat m_{ijk} \in \Z^{n}$  and smooth maps
\begin{equation*}
a_{ij},\hat a_{ij}: U_i \cap U_j \to \R^{n}
\quomma
t_{ijk}:U_i \cap U_j \cap U_k \to \ueins
\end{equation*}
satisfying the following cocycle conditions:
\begin{align*}
a_{ik} &= m_{ijk} + a_{jk}+a_{ij}
\\
\hat a_{ik} &= \hat m_{ijk} + \hat a_{jk}+\hat a_{ij}
\\
m_{ikl}+m_{ijk} &= m_{ijl}+m_{jkl}
\\
\hat m_{ikl}+\hat m_{ijk} &= \hat m_{ijl}+\hat m_{jkl}
\\
t_{ikl}+t_{ijk}-m_{ijk}\hat a_{kl} &= t_{ijl}+t_{jkl}\text{.}
\end{align*}
Here, the expression $m_{ijk}\hat a_{kl}$ denotes the standard scalar product, and $\ueins = \R/\Z$ is written additively as before. Moreover, we have tacitly assumed that the 3-fold intersections of our open cover are connected; in general, the numbers $m_{ijk}$ would be $\Z$-valued functions. 

For the definition of T-duality correspondences we refer to \cite[§3.1]{Nikolause}; these are equivalent (\cite[Prop. 3.11]{Nikolause}) to the one of Bunke-Rumpf-Schick \cite{bunke2006a}, which are in turn equivalent to the one of Mathai-Rosenberg \cite{Mathai2006}.
The relation between $\TD[n]$-cocycles and T-duality correspondences is described in \cite{Nikolause} only in an indirect way from a stack-theoretical perspective. Below we give a direct construction; a more elaborate discussion on these matters will appear elsewhere.

From a $\TD[n]$-cocycle as above we  reconstruct principal $\T^{n}$-bundles $E$ and $\hat E$ over $X$ with transition functions $e^{2\pi \im a_{ij}}$ and $e^{2\pi \im \hat a_{ij}}$, respectively. Note that these come with canonical trivializations $\varphi_i$ and $\hat\varphi_i$ over the open sets $U_i$. 
The next step is to construct bundle gerbes $\mathcal{G}$ and $\widehat{\mathcal{G}}$ over $E$ and $\hat E$, respectively. We consider 
\begin{equation*}
Y := \coprod_{i\in I} U_i \times \T^{n}
\end{equation*}
and define the surjective submersions $\pi:Y \to E$ and $\hat\pi: Y \to \hat E$ by $\pi(i,x,a):=\varphi_i(x,a)$ and $\hat\pi(i,x,a):=\hat\varphi_i(x,a)$. We may then identify the fibre products of $\pi:Y\to E$ with 
\begin{equation*}
Y^{[k]} \cong \coprod_{(i_1,...,i_k) \in I^{k}} Y_{i_1,...,i_k}
\quand
Y_{i_1,...,i_k} := (U_{i_1}\cap ...\cap U_{i_k} )\times \T^{n}\text{,} 
\end{equation*}
by sending $((i_1,x,a_1),...,(i_k,x,a_k)) \in Y^{[k]}$ to $(x,a_1)\in Y_{i_1,...,i_k}$. This establishes a diffeomorphism as the fibre product condition infers $a_j = a_1+a_{i_1i_j}(x)$.
Thus, under this identification, the projection maps $\pr_j: Y^{[k]} \to Y$ are given by
\begin{equation*}
\pr_j|_{Y_{i_1,...,i_k}}(x,a) := (i_j,x,a + a_{i_1i_j}(x))\text{.}
\end{equation*}
The fibre products of $\hat \pi:Y \to \hat E$ will be identified with the same manifold, but the projection maps are here given by  
\begin{equation*}
\hat{\pr}_j|_{Y_{i_1,...,i_k}}(x,\hat a) := (i_j,x,\hat a + \hat a_{i_1i_j}(x))\text{.}
\end{equation*}
Now, cocycles $\beta,\hat\beta: Y^{[3]} \to \ueins$ w.r.t. the surjective submersions $\pi$ and $\hat\pi$, respectively, are defined by
\begin{align}
\label{eq:cocycle:lele}
\beta_{ijk} &: Y_{ijk} \to \ueins: (x,a) \mapsto -t_{ijk}(x)-a\hat m_{ijk}+a_{ij}(x)\hat a_{jk}(x)
\\
\label{eq:cocycle:rile}
\hat \beta_{ijk} &: Y_{ijk} \to \ueins: (x,\hat a) \mapsto -t_{ijk}(x)-m_{ijk}(\hat a_{ik}(x)+\hat a)\text{.}
\end{align}
One may check explicitly using the $\TD[n]$-cocycle conditions and above formulas for the projections that $\beta$ and $\hat\beta$ satisfy the cocycle conditions. Hence, we have defined bundle gerbes $\mathcal{G}$ over $E$ and $\widehat{\mathcal{G}}$ over $\hat E$.

It remains to construct a bundle gerbe isomorphism $\mathcal{D}: \pr^{*}\mathcal{G} \to \hat\pr^{*}\widehat{\mathcal{G}}$ over the correspondence space $E \times_X \hat E$, establishing T-duality.  For this purpose, we look at the diagram
\begin{equation*}
\alxydim{}{& Z \ar[d]^{\zeta} \ar[dl]_{\pr'} \ar[dr]^{\hat\pr'} \\ Y \ar[d]_{\pi} & E \times_X \hat E \ar[dr]^{\hat \pr} \ar[dl]_{\pr} & Y \ar[d]^{\hat \pi} \\ E \ar[dr]_{p} && \hat E \ar[dl]^{\hat p} \\ & X}
\end{equation*}
where 
\begin{equation*}
Z := \coprod_{i\in I}  U_i \times \R^{2n}\text{,}
\end{equation*}
and the maps $\zeta$, $\pr'$, and $\hat\pr'$ are defined by
$\zeta (i,x,a,\hat a) := (\varphi_i(x,a),\hat\varphi_i(x,\hat a))$, $\pr' (i,x,a,\hat a) := (i,x,a)$, and $\hat\pr' (i,x,a,\hat a) := (i,x,\hat a)$. The commutativity of the diagram shows that the pullbacks of the bundle gerbes $\mathcal{G}$ and $\widehat{\mathcal{G}}$ to the correspondence space $E \times_X \hat E$ can be represented by cocycles w.r.t. the surjective submersion $\zeta$. 

The fibre products of $\zeta:Z \to E \times_X \hat E$ can be written as
\begin{equation*}
Z^{[k]} \cong \coprod_{i_1,...,i_k} Z_{i_1,...,i_k}
\quith Z_{i_1,...,i_k}:= U_{i_1}\cap ... \cap U_{i_k}  \times \R^{2n} \times \underbrace{\Z^{2n}\times ... \times \Z^{2n}}_{(k-1)\text{ times}}
\end{equation*}
in such a way that an element $((i_1,x,a_1,\hat a_1),...,(i_k,x,a_k,\hat a_k))\in Z^{[k]}$ corresponds to 
\begin{equation*}
(i_1,...,i_k,x,a_1,\hat a_1,m_2,\hat m_2,...,m_k,\hat m_k)\in Z_{i_1,...,i_k}
\end{equation*}
where the integers  are defined by equations $a_{p}=a_{1}+a_{i_1i_p}(x)+m_{p}$ and $\hat a_{p}=\hat a_{1}+\hat a_{i_1i_p}(x)+\hat m_{p}$ for $2\leq p \leq k$. Under this identification, the projections $\pr_j: Z^{[k]} \to Z$ are given by
\begin{align*}
&\pr_j(i_1,...,i_k,x,a,\hat a,m_2,\hat m_2,...,m_k,\hat m_k):=(i_{j},x,a+a_{i_1i_{j}}(x)+m_{j},\hat a+\hat a_{i_1i_{j}}(x)+\hat m_{j})\text{.}
\end{align*}
We define $\xi:Z^{[2]} \to \ueins$ by
\begin{equation}
\xi_{ij}:Z_{ij} \to \ueins: (i,j,x,a,\hat a,m_2,\hat m_2) \mapsto -m_2\hat a- \hat a_{ij}(x) m_2-  \hat a_{ij}(x)a\text{.}
\end{equation}
One can now check that over $Z^{[3]}$ the equality
\begin{equation*}
\hat\pr'^{*}\hat\beta - \pr'^{*}\beta = \pr_{12}^{*}\xi +\pr_{23}^{*}\xi - \pr_{13}^{*}\xi 
\end{equation*}
holds, thus defining the required isomorphism $\mathcal{D}$.
Finally, it remains to check that $\mathcal{D}$ satisfies the \quot{Poincaré condition}. Indeed, upon restricting to a single open set $U_i$ (equipped with the cover by a single open set, itself), the 3-cocycles \cref{eq:cocycle:lele,eq:cocycle:rile} become identically zero (we may assume as usual that the data $t_{ijk}$, $m_{ijk}$, $\hat m_{ijk}$, $a_{ij}$, and $\hat a_{ij}$ of a $\TD[n]$-cocycle are  all anti-symmetric in the indices), and the 2-cochain $\xi$ becomes $(i,j,x,a,\hat a,m_2,\hat m_2)  \mapsto -m_2\hat a$. This is, in particular, independent of $x$ and thus is the pullback along $Z_i \to \T^{2n}$ of a 2-cocycle on $\T^{2n}$  w.r.t. to the surjective submersion $\R^{2n} \to \T^{2n}$. That 2-cocycle, $(a,\hat a,m,\hat m)  \mapsto -m\hat a$ is known to represent the Poincaré bundle over $\T^{2n}$. This completes the construction of a T-duality correspondence from a $\TD[n]$-cocycle.

\subsection{Geometric interpretation of some automorphisms}

\label{sec:actions}

Crossed  intertwiners $\Gamma \to \Gamma'$ induce maps $\h^1(X,\Gamma) \to \h^1(X,\Gamma)$ in cohomology, which can be represented explicitly on the level of cocycles, see \cite[§A.3]{Nikolause}. In case of a crossed intertwiner
\begin{equation*}
(\phi,f,\eta):\TD[n] \to \TD[n]
\end{equation*}
and a $\TD[n]$-cocycle $(a,\hat a,m,\hat m,t)$, the resulting $\TD[n]$-cocycle $(a',\hat a',m',\hat m', t')$ is given by the formulas 
\begin{align*}
a'_{ij} \oplus \hat a'_{ij} &:= \phi (a_{ij}\oplus \hat a_{ij})
\\
m'_{ijk}\oplus m_{ijk} &:= f_1(m_{ijk}\oplus \hat m_{ijk},t_{ijk})
\\
t'_{ijk} &:=  f_2(m_{ijk}\oplus \hat m_{ijk},t_{ijk})
\\&\qquad-\eta(m_{ijk}\oplus\hat m_{ijk},(a_{jk}+ a_{ij})\oplus (\hat a_{jk}+ \hat a_{ij}))-\eta(a_{jk}\oplus \hat a_{jk},a_{ij} \oplus \hat a_{ij})\text{;}
\end{align*} 
here, we have split $f=:(f_1,f_2)$ into maps $f_1: \Z^{2n} \times \ueins \to \Z^{2n}$ and $f_2: \Z^{2n} \times \ueins \to \ueins$.

These formula simplify slightly when we start with an object
$(A,\eta)$ in $\AU[n]$ and consider the crossed intertwiner  $\mathbb{I(A,\eta)}$ described in \cref{sec:realizationauto}, i.e., $\mathbb{I}(A,\eta)=(\phi,f,\eta)$ with  $\phi(a \oplus \hat a) := A(a\oplus \hat a)$ and   $f_{1}(m\oplus \hat m,t):= A(m\oplus \hat m)$, and $f_2(m\oplus \hat m,t):= \mathrm{iso}(A)t$. The resulting $\TD[n]$-cocycle is then 
\begin{align}
\label{eq:formulaaction:1}
a'_{ij} \oplus \hat a'_{ij} &:=A (a_{ij}\oplus \hat a_{ij})
\\
\label{eq:formulaaction:2}
m'_{ijk}\oplus m_{ijk} &:= A(m_{ijk}\oplus \hat m_{ijk})
\\\nonumber
t'_{ijk} &:= \mathrm{iso}(A)t_{ijk} -\eta(m_{ijk}\oplus\hat m_{ijk},(a_{jk}+ a_{ij})\oplus (\hat a_{jk}+ \hat a_{ij}))-\eta(a_{jk}\oplus \hat a_{jk},a_{ij} \oplus \hat a_{ij})\text{.}
\end{align} 

We can go further and use the section $S: \mathrm{O}^{\pm}(n,n,\Z) \to \AU[n]$ defined in \cref{lem:section} in order to get a crossed intertwiner $\mathbb{I}(S(A))$ associated to each $A\in \mathrm{O}^{\pm}(n,n,\Z)$ in a canonical way. The formulas are the same as above, just that the general $\eta$ is replaced by $\eta_A$ defined in \cref{def:etaA}. Since $\eta_A$ is bilinear and satisfies \cref{AU:1}, the formula for $t'_{ijk}$ simplifies slightly to
\begin{equation}
\label{eq:formulaaction:3}
t'_{ijk} := \mathrm{iso}(A)t_{ijk} -\eta_A(m_{ijk}\oplus\hat m_{ijk},a_{ik}\oplus \hat a_{ik})-\eta_A(a_{jk}\oplus \hat a_{jk},a_{ij} \oplus \hat a_{ij})\text{.}
\end{equation}

One can now reconstruct from the original cocycle $(a,\hat a,m,\hat m,t)$ and from the new cocycle $(a',\hat a',m',\hat m',t')$ topological T-duality correspondences as described in \cref{sec:class} and  observe the geometric effect. For a general element $A\in \mathrm{O}^{\pm}(n,n,\Z)$
this seems impossible, but below we manage to describe this effect in those situations where the k-invariant of $\AUT(\TD[n])$ vanishes according to our calculations in \cref{sec:splitting:multiplicativity}. First, we look at the subgroup $Z=\left \langle I  \right \rangle\cong \Z/2\Z$ of $\mathrm{O}(n,n,\Z)$ introduced in \cref{subgroup-Z}, and its action on $\TD_n$ induced along the   2-group homomorphism $\mathbb{I}_S:Z \to \AUT(\TD_n)$ of \cref{prop:actions}.

\begin{proposition}
\label{prop:actionI}
The action of $Z \cong \Z/2\Z$ on $\TD[n]$ flips the two legs of a T-duality correspondence.   
\end{proposition}

\begin{proof}
We remark that we have already seen in \cref{ex:flipautomorphism} that the crossed intertwiner  $\mathbb{I}(S(I))$ coincides with the crossed intertwiner called  $flip$ of \cref{ex:flip}. Formulas \cref{eq:formulaaction:1,eq:formulaaction:2,eq:formulaaction:3} show that
$a'_{ij}=\hat a_{ij}$, $\hat a'_{ij} = a_{ij}$, $m'_{ijk} = \hat m_{ijk}$, and $\hat m_{ijk}' = m_{ijk}$,
as well as
\begin{align*}
t'_{ijk} &= t_{ijk} - [m_{ijk}\oplus\hat m_{ijk},a_{ik}\oplus \hat a_{ik}]- [a_{jk}\oplus \hat a_{jk},a_{ij} \oplus \hat a_{ij}]
=t_{ijk} - \hat m_{ijk}a_{ik}-  \hat a_{jk}a_{ij}\text{.} 
\end{align*}
It is hence clear that the torus bundles $E$ and $\hat E$ become exchanged. The  bundle gerbe cocycle for the new left leg is, according to \cref{eq:cocycle:lele},
\begin{align*}
\beta_{ijk}'(x,a) &=-t'_{ijk}(x)-a\hat m_{ijk}'+a_{ij}'(x)\hat a_{jk}'(x)
\\&= -t_{ijk} (x)+ \hat m_{ijk}a_{ik}(x)+  \hat a_{jk}(x)a_{ij} (x)-am_{ijk}+\hat a_{ij}(x)a_{jk}(x)
 \\&= \hat \beta_{ijk}(x,a)-a_{ij}(x)\hat a_{ij}(x)-a_{jk}(x)\hat a_{jk}(x)+a_{ik}(x)\hat a_{ik}(x)\text{,} 
\end{align*}
and it hence differs from $\hat\beta_{ijk}$ by the coboundary of $(x,a)\mapsto a_{ij}(x)\hat a_{ij}(x)$. The bundle gerbe cocycle for the new right leg is, according to \cref{eq:cocycle:rile},
\begin{align*}
\hat \beta_{ijk}' (x,\hat a) &=-t'_{ijk}(x)-m_{ijk}'(\hat a_{ik}'(x)+\hat a)
\\&= -t_{ijk} (x)+ \hat m_{ijk}a_{ik}(x)+  \hat a_{jk}(x)a_{ij} (x)-\hat m_{ijk}(a_{ik}(x)+\hat a)
\\&= \beta_{ijk}(x,\hat a)
\end{align*}
This shows that the bundle gerbes are exchanged, too.
\end{proof}

The next result concerns the subgroup $\mathrm{GL}(n,\Z)$ of $\mathrm{O}(n,n,\Z)$ introduced in \cref{subgroup-GLnZ}, and its action on $\TD_n$ induced along the 2-group homomorphism $\mathbb{I}_S: \mathrm{GL}(n,\Z) \to \AUT(\TD_n)$ of \cref{prop:actions}.

\begin{proposition}
\label{prop:actionofGL}
The action of $A\in\mathrm{GL}(n,\Z)$ on $\TD[n]$ extends the torus bundles along the corresponding Lie group homomorphisms $A:\T^{n} \to \T^{n}$ and $A^{tr-1}:\T^{n} \to \T^{n}$, respectively, and pushes the bundle gerbes to these extensions.
\end{proposition}

\begin{proof}
For $A\in \mathrm{GL}(n,\Z)$ and the corresponding $D_A\in \mathrm{O}(n,n,\Z)$ we have $\eta_{D_A}=0$ by \cref{ex:split:d}. Hence,  \cref{eq:formulaaction:1,eq:formulaaction:2,eq:formulaaction:3} give
\begin{equation*}
a'_{ij} = Aa_{ij}
\quomma 
\hat a'_{ij} =A^{tr-1}\hat a_{ij}
\quomma
m'_{ijk}=Am_{ijk}
\quomma
\hat m'_{ijk}=A^{tr-1}\hat m_{ijk}
\quomma
t'_{ijk}=t_{ijk}\text{.}
\end{equation*}
We may consider $A$ as a Lie group isomorphism $A: \T^{n} \to \T^{n}$, and observe that the new torus bundles are \quot{extensions} along $A$ and $A^{tr-1}$ of the old ones, respectively, i.e.,
\begin{equation*}
E' = A_{*}(E)
\quand
\hat E' =A^{tr-1}_{*}(\hat E)\text{.}
\end{equation*} 
Note that such extensions come with diffeomorphisms $\alpha:E \to E'$ and $\hat\alpha: \hat E\to \hat E'$ that are equivariant along $A$ and $A^{tr-1}$, respectively, and respect the bundle projections to $X$. Also note that we have commutative diagrams
\begin{equation*}
\alxydim{@C=4em}{Y \ar[r]^{\id \times A} \ar[d]_{\pi} & Y \ar[d]^{\pi'} \\ E \ar[r]_{\alpha} & E'}
\quand
\alxydim{@C=4em}{Y \ar[r]^{\id \times A^{tr-1}} \ar[d]_{\hat \pi} & Y \ar[d]^{\hat \pi'} \\ \hat E \ar[r]_{\hat \alpha} & \hat E'\text{.}}
\end{equation*} 
The bundle gerbe cocycle for the new left leg is  according to \cref{eq:cocycle:lele}
\begin{align*}
\beta_{ijk}'(x,a) &=-t_{ijk}(x)-a(A^{tr-1}\hat m_{ijk})+(Aa_{ij}(x))(A^{tr-1}\hat a_{jk}(x))
\\&=-t_{ijk}(x)-(A^{-1}a)\hat m_{ijk}+a_{ij}(x)\hat a_{jk}(x)
\\&=\beta_{ijk}(x,A^{-1}a)\text{.}
\end{align*}
The bundle gerbe cocycle for the new right leg is according to \cref{eq:cocycle:rile}
\begin{align*}
\hat \beta_{ijk}' (x,\hat a) &=-t_{ijk}(x)-(Am_{ijk})(A^{tr-1}\hat a_{ik}'(x)+\hat a)
\\&= -t_{ijk}(x)-m_{ijk}(a_{ik}'(x)+A^{tr}\hat a)
\\&= \hat \beta_{ijk}(x,A^{tr}\hat a)\text{.}
\end{align*}
This means that $\mathcal{G} = \alpha^{*}\mathcal{G}'$ and $\widehat{\mathcal{G}} = \hat\alpha^{*}\widehat{\mathcal{G}}'$.
\end{proof}

We come to the Klein-4-subgroup $K$ of $\mathrm{O}^{\pm}(n,n,\Z)$ introduced in \cref{subgroup-K}, which acts again on $\TD_n$ along the 2-group homomorphism $\mathbb{I}_S: \mathrm{O}^{\pm}(n,n,\Z) \to \AUT(\TD_n)$ of \cref{prop:actions}.

\begin{proposition}
Under the action of $K$ on $\mathbb{TD}_n$, the generator $K_{-1,1}$ dualizes the left leg's torus bundle and dualizes both bundle gerbes, while the generator $K_{1,-1}$ dualizes the right leg's torus bundle and  dualizes both bundle gerbes. 
\end{proposition}

\begin{proof}
We have $\eta_{K_{\alpha,\hat\alpha}}=0$ by \cref{ex:split:g}. Hence,  \cref{eq:formulaaction:1,eq:formulaaction:2,eq:formulaaction:3} give
\begin{equation*}
a'_{ij} = \alpha a_{ij}
\quomma 
\hat a'_{ij} =\hat\alpha \hat a_{ij}
\quomma
m'_{ijk}=\alpha m_{ijk}
\quomma
\hat m'_{ijk}=\hat\alpha \hat m_{ijk}
\quomma
t'_{ijk}=\alpha\hat\alpha t_{ijk}\text{.}
\end{equation*}
Changing the sign of transitions functions corresponds to passing to the dual bundle; hence, $E' := E^{\alpha}$ and $\hat E':= E^{\hat\alpha}$. As in the proof of \cref{prop:actionofGL}, we have commutative diagrams 
\begin{equation*}
\alxydim{@C=4em}{Y \ar[r]^{\id \times \alpha} \ar[d]_{\pi} & Y \ar[d]^{\pi'} \\ E \ar@{=}[r] & E'}
\quand
\alxydim{@C=4em}{Y \ar[r]^{\id \times \hat\alpha} \ar[d]_{\hat \pi} & Y \ar[d]^{\hat \pi'} \\ \hat E \ar@{=}[r] & \hat E'\text{.}}
\end{equation*} 
\cref{eq:cocycle:lele,eq:cocycle:rile} lead to the formulas $\beta'_{ijk}(x,a)=\alpha\hat\alpha \beta_{ijk}(x,\alpha a)$ and $\hat\beta'_{ijk}(x,\hat a) = \alpha\hat\alpha \hat\beta_{ijk}(x,\hat\alpha\hat a)$. This means that $\mathcal{G}' := \mathcal{G}^{\alpha\hat\alpha}$ and $\hat{\mathcal{G}}':= \hat{\mathcal{G}}^{\alpha\hat\alpha}$. Here, $\mathcal{G}^{-1}$ denotes the dual bundle gerbe, while $\mathcal{G}^1=\mathcal{G}$. \end{proof}

Now we treat the special case of $n=1$, where $\mathrm{O}^{\pm}(1,1,\Z)=\left \langle I,R  \right \rangle \cong D_4$ is a dihedral group generated by one reflection, $I$, and one rotation $R$ (see \cref{ex:n=1}). While the reflection $I$ flips a T-duality correspondence just like for general $n$ (\cref{prop:actionI}), the rotation acts as follows.   

\begin{proposition}
Under the action of $\mathrm{O}^{\pm}(1,1,\Z)=\left \langle I,R  \right \rangle \cong D_4$ on $\TD[1]$, the rotation
\begin{equation*}
R= \begin{pmatrix}0 & -1 \\
1 & 0 \\
\end{pmatrix}
\end{equation*}
flips the two legs, dualizes the (new) left leg circle bundle, and dualizes both bundle gerbes. 
\end{proposition}

\begin{proof}
We have $\mathrm{iso}(R)=-1$ and $\eta_R = -[-,-]$ by \cref{ex:split:b}.
 \cref{eq:formulaaction:1,eq:formulaaction:2,eq:formulaaction:3} now give
$a'_{ij}=-\hat a_{ij}$, $\hat a'_{ij} = a_{ij}$, $m'_{ijk} =- \hat m_{ijk}$, and $\hat m_{ijk}' = m_{ijk}$
as well as
\begin{align*}
t'_{ijk} &= -t_{ijk} + \hat m_{ijk}a_{ik}+  \hat a_{jk}a_{ij}\text{.} 
\end{align*}
Thus, we have for the torus bundles $\hat E' = E$ and $E' = \hat E^{*}$, the dual principal bundle. The bundle gerbe cocycle for the new left leg is,  according to \cref{eq:cocycle:lele},
\begin{align*}
\beta_{ijk}'(x,a) &=-t'_{ijk}(x)-a\hat m_{ijk}'+a_{ij}'(x)\hat a_{jk}'(x)
\\&= t_{ijk} (x)- \hat m_{ijk}a_{ik}(x)-  \hat a_{jk}(x)a_{ij}(x)-am_{ijk}-\hat a_{ij}(x)a_{jk}(x)
 \\&= -\hat\beta_{ijk}(x,-a)-m_{ijk}\hat a_{ik}(x)- \hat m_{ijk}a_{ik}(x)-  \hat a_{jk}(x)a_{ij}(x)-\hat a_{ij}(x)a_{jk}(x)
 \\&=-\hat\beta_{ijk}(x,-a) +a_{ij}(x)\hat a_{ij}(x)+a_{jk}(x)\hat a_{jk}(x)-a_{ik}(x)\hat a_{ik}(x)\text{;}
\end{align*}
hence, up to a coboundary (see the proof of \cref{prop:actionI}), we have $\beta_{ijk}'(x,a) = -\hat\beta_{ijk}(x,-a)$.
The bundle gerbe cocycle for the new right leg is, according to \cref{eq:cocycle:rile},
\begin{align*}
\hat \beta_{ijk}' (x,\hat a) &=-t'_{ijk}(x)-m_{ijk}'(\hat a_{ik}'(x)+\hat a)
\\&= t_{ijk} - \hat m_{ijk}a_{ik}(x)-  \hat a_{jk}(x)a_{ij}(x)+\hat m_{ijk}(a_{ik}(x)+\hat a) 
\\&= t_{ijk} -  \hat a_{jk}(x)a_{ij}(x)+\hat m_{ijk}\hat a
\\&= -\beta_{ijk}(x,\hat a)\text{.} 
\end{align*}
We recall that the dual of a principal bundle has the same total space but the action is through inverses. 
We hence have commutative diagrams
\begin{equation*}
\alxydim{@C=4em}{Y \ar[r]^{\id} \ar[d]_{\hat \pi'} & Y \ar[d]^{\pi} \\ \hat E' \ar@{=}[r] & E}
\quand
\alxydim{@C=4em}{Y \ar[r]^{d} \ar[d]_{\pi'} & Y \ar[d]^{\hat \pi'} \\ E' \ar@{=}[r] & \hat E^{*}}
\end{equation*} 
in which the map $d$ is $d_i(x,a) := (x,-a)$. Thus, above cocycle equation mean that $\mathcal{G}' := \widehat{\mathcal{G}}^{*}$ and $\widehat {\mathcal{G}}' = \mathcal{G}^{*}$.
\end{proof}

Finally, we consider the  subgroup $\mathfrak{so}(n,\Z)$ of $\mathrm{O}(n,n,\Z)$ on $\TD_n$ (see \cref{subgroup-sonZ}). It acts again on $\TD_n$ via the 2-group homomorphism $\mathbb{I}_S: \mathfrak{so}(n,\Z) \to \AUT(\TD_n)$ of \cref{prop:actions}. In  \cite{Nikolause} such action was found and described \quot{manually}, and in \cite{bunke2006a}, a similar action of  $\mathfrak{so}(n,\Z)$ on T-duality triples was found.  

\begin{proposition}
\label{prop:inducedaction}
The action of $B\in \mathfrak{so}(n,\Z)$ on $\TD[n]$ coincides with the one considered in  \cite{Nikolause}, and moreover reproduces the action of $\mathfrak{so}(n,\Z)$ on T-duality triples considered in \cite{bunke2006a}. It fixes the left leg torus bundle $E$, and fixes  the left leg bundle gerbe $\mathcal{G}$ up to a canonical isomorphism. The underlying torus bundle of the new right leg is
\begin{equation*}
\hat E' = \hat E \otimes B_{*}(E)\text{,}
\end{equation*} 
where $B$ is considered (via matrix multiplication) as a group homomorphism $B:\T^{n} \to \T^{n}$, and $B_{*}(E)$ denotes the bundle extension along $B$.
Finally, the bundle gerbe $\widehat{\mathcal{G}}'$ of the new right leg is a pull-push of $\widehat{\mathcal{G}}$ along the span
\begin{equation*}
\alxydim{}{& \hat E \times_X B_{*}(E) \ar[dr]\ar[dl]_{\pr_{\hat E}} \\ \hat E && \hat E'}
\end{equation*} 
whose right leg is the canonical map $\hat E \times_X B_{*}(E) \to \hat E \otimes B_{*}(E)$ from the fibre product to the tensor product of abelian principal bundles. 
\end{proposition}

\begin{proof}
For $B \in \mathfrak{so}(n,\Z)$ we have $S(B)=(e^{B},\eta_{e^{B}})$ in $\AU[n]$, where $\eta_{e^{B}}(a \oplus \hat a,b \oplus \hat b) = \lbraket{a}{B}{b}$ according to \cref{ex:split:e}. The crossed intertwiner $\mathbb{I}(S(B)) =: (\phi_{e^{B}},f_{e^{B}},\eta_{e^{B}})$ is then given by \begin{equation*}
\phi_{e^{B}}(a \oplus \hat a) = e^{B}(a \oplus \hat a)=a \oplus (Ba+\hat a)
\end{equation*}
and 
\begin{equation*}
f_{e^{B}}(m\oplus \hat m,s)= (e^{B}(m\oplus \hat m),s)=(m\oplus (Bm+\hat m),s)\text{.}
\end{equation*}
Precisely this crossed intertwiner is defined in \cite[Sec. 4.1]{Nikolause} and proves the claimed coincidence.
 \cref{eq:formulaaction:1,eq:formulaaction:2,eq:formulaaction:3}  give  
$a_{ij}'=a_{ij}$, $\hat a_{ij}' = Ba_{ij}+\hat a_{ij}$, $m'_{ijk}=m_{ijk}$, and $\hat m_{ijk}'=Bm_{ijk}+\hat m_{ijk}$, as well as
\begin{align*}
t'_{ijk} &:=  t_{ijk}-\lbraket{m_{ijk}}{B}{a_{ik}}-\lbraket{a_{jk}}{B}{a_{ij}} \text{.}
\end{align*} 
We see that the left leg torus bundle is unchanged, while the transition functions of the new right leg
torus bundle show the claim that $\hat E' = B_{*}(E) \otimes \hat E$. 
For the $n$ Chern classes $\hat c'_i$ of $\hat E'$ and $\hat c_i$ of $\hat E$, this implies the formula
\begin{equation*}
\hat c_i' = \hat c_i + \sum_{j=1}^{n}B_{ij} c_j\text{,}
\end{equation*}
which appears in \cite[Thm. 2.24]{bunke2006a}. This shows the next claimed coincidence. It remains to look at the bundle gerbes.

For the left leg, a computation using \cref{eq:cocycle:lele} shows that
\begin{equation*}
\beta_{ijk}'(x,a)= \beta_{ijk}(x,a)+\lbraket{m_{ijk}}{B}{a_{ik}}+\lbraket{a_{ij}}{B}{a_{jk}}-a(Bm_{ijk})
\end{equation*}
One can show that the three last terms form a coboundary, showing that $\mathcal{G}'\cong \mathcal{G}$. Since this also follows abstractly from \cite[Lemma 4.3]{Nikolause}, we can omit an explicit discussion. For the right leg, another computation using \cref{eq:cocycle:rile} shows that
\begin{equation*}
\hat \beta_{ijk}' (x,\hat a) = -t_{ijk}(x)-m_{ijk}(\hat a_{ik}(x)+a)+\lbraket{a_{ik}}{B}{m_{ijk}}+\lbraket{a_{jk}}{B}{a_{ij}}\text{.}
\end{equation*}
We consider the commutative diagram
\begin{equation*}
\alxydim{}{Y \ar[d] & \coprod U_i \times \T^{2n} \ar[r] \ar[l] \ar[d] & Y \ar[d] \\ \hat E  & \hat E \times_X B_{*}(E) \ar[r]\ar[l] & \hat E \otimes B_{*}(E)\text{,}}
\end{equation*}
which lifts the span from the claim of the proposition to the surjective submersions of bundle gerbes, 
where the maps on top are given by
\begin{equation*}
\alxydim{}{(x,i,a) & (i,x,a,b) \ar@{|->}[r] \ar@{|->}[l] & (x,i,a+b)\text{.}}
\end{equation*}
We may now analyze the pullbacks of the cocycles $\hat\beta_{ijk}$ for $\widehat{\mathcal{G}}$ and $\hat\beta'_{ijk}$ for $\widehat{\mathcal{G}}'$ to the middle. In order to prove the claim that $\widehat{\mathcal{G}}'$ is a pull-push of $\widehat{\mathcal{G}}$ we need to show that the difference of these pullbacks,
\begin{equation*}
\tilde\gamma_{ijk}(x,a,b) := \hat\beta'_{ijk}(x,a+b) - \hat\beta_{ijk}(x,a)\text{,}
\end{equation*}
is a coboundary. A computation shows that
\begin{equation*}
\tilde\gamma_{ijk}(x,a,b)=  \lbraket{a_{ik}(x)}{B}{m_{ijk}}+\lbraket{a_{jk}(x)}{B}{a_{ij}(x)} -m_{ijk}b\text{,}
\end{equation*}
which is independent of $a$. 
In fact, $\gamma_{ijk}(x,b) := \tilde\gamma_{ijk}(x,0,b)$ defines a cocycle for a bundle gerbe over $B_{*}(E)$, pulling back to $\tilde\gamma_{ijk}$ under the projection $E \times_X B_{*}(E) \to B_{*}(E)$, and now we have to show that $\gamma_{ijk}$ is a coboundary.
For this purpose, a further computation shows that
\begin{equation*}
\gamma_{ijk}(x,b) = \tilde\alpha_{ij}(x,b)+\tilde\alpha_{jk}(x,b+Ba_{ij}(x))-\tilde\alpha_{ik}(x,b)+\varepsilon_{ijk}(x)\text{,}
\end{equation*} 
for $\tilde\alpha_{ij}(x,b) := a_{ij}(x)b$ and 
\begin{align*}
\varepsilon_{ijk}(x) &:= \lbraket{a_{ik}(x)}{B}{m_{ijk}}+\lbraket{a_{ij}(x)}{B}{a_{jk}(x)}\text{.}
\end{align*}
One can check that $\varepsilon_{ijk}$ is a smooth $\R$-valued \v Cech 3-cocycle on $X$, and hence trivializable as the sheaf $\underline{\R}$ is fine. 
All together, this proves that $\gamma_{ijk}$ is a coboundary.
\end{proof}

\bibliographystyle{kobib}
\bibliography{kobib}

\end{document}

%% file: styles.tex
\usepackage{amsthm}
\usepackage{thmtools}
\usepackage{graphicx}
\usepackage[margin=2cm,font=normalsize,labelfont=bf]{caption}
\usepackage{verbatim}
\usepackage[shortlabels]{enumitem}
\usepackage{fancyhdr}
\usepackage[]{geometry}
\usepackage{xstring}
\usepackage{needspace}
\usepackage{color}
\usepackage{amstext}
\usepackage{nameref}
\usepackage[hypertexnames=false]{hyperref}
\usepackage{mathdots}
\usepackage{soul}
\usepackage{chngcntr}
\usepackage[section]{placeins}

\makeatletter
\newcounter{denseversion}
\newcounter{comments}
\newcounter{authorcounter}
\newcounter{adresscounter}

\def\title#1{\gdef\@title{#1}}
\def\@title{}
\def\subtitle#1{\gdef\@subtitle{#1}}
\def\@subtitle{}

\def\authortagsused{0}
\def\adresstag#1{\if!#1!\else$^{\;#1\;}$\fi}
\def\@authorsep#1{
  \ifnum\value{authorcounter}=#1 and \else\unskip, \fi
}

\renewcommand{\author}[2][]{
  \stepcounter{authorcounter}
  \if!#1!\else\gdef\authortagsused{1}\fi
  \ifnum\value{authorcounter}=1
    \def\@authorstringa{#2\adresstag{#1}}
    \def\@authorstringb{#2}
    \def\@authorstringc{#2\adresstag{#1}}
  \else
    \ifnum\value{authorcounter}=2
      \g@addto@macro\@authorstringa{\@authorsep{2}#2\adresstag{#1}}
      \g@addto@macro\@authorstringb{\@authorsep{2}#2}
      \g@addto@macro\@authorstringc{\\#2\adresstag{#1}}
    \else
      \g@addto@macro\@authorstringa{\@authorsep{3}#2\adresstag{#1}}
      \g@addto@macro\@authorstringb{\@authorsep{3}#2}
      \g@addto@macro\@authorstringc{\\#2\adresstag{#1}}
    \fi
  \fi}
\def\@author{\ifnum\value{denseversion}=0\@authorstringa\else\@authorstringb\fi}

\def\@adressstringa{}
\def\@adressstringb{}
\newcommand{\adress}[2][]{
  \stepcounter{adresscounter}
  \ifnum\value{adresscounter}=1
    \g@addto@macro\@adressstringa{\ifnum\authortagsused=0\def\br{\\}\else\def\br{, }\fi\adresstag{#1}#2}
    \g@addto@macro\@adressstringb{\def\br{\\}\adresstag{#1}\parbox[t]{14cm}{#2}}
  \else
    \g@addto@macro\@adressstringa{\\[\bigskipamount]\adresstag{#1}#2}
    \g@addto@macro\@adressstringb{\\[\medskipamount]\adresstag{#1}\parbox[t]{14cm}{#2}}
  \fi}

\def\preprint#1{\gdef\@preprint{#1}}
\def\@preprint{}
\def\keywords#1{\gdef\@keywords{#1}}
\def\@keywords{}
\def\msc#1{\gdef\@msc{#1}}
\def\@msc{}
\def\email#1{
   \gdef\@email{#1}
   \g@addto@macro\@authorstringc{ {\it (#1)}}}
\def\@email{}
\def\dedication#1{\gdef\@dedication{#1}}
\def\@dedication{}

\def\mybaselinestretch#1{
  \gdef\@mybaselinestretch{#1}
  \renewcommand{\baselinestretch}{\@mybaselinestretch}}

\def\myparskip#1{
  \gdef\@myparskip{#1}
  \setlength{\parskip}{\@myparskip}}

\mybaselinestretch{1.2}
\myparskip{1.5ex}

\binoppenalty=\maxdimen
\relpenalty=\maxdimen

\normalfont
\normalsize
\newlength{\@listleftmargin}
\def\setenumstandard{
  \setlist{leftmargin=\@listleftmargin,itemsep=0pt,topsep=0pt,partopsep=0pt,parsep=\@myparskip}
  \setlist[enumerate]{align=left,labelsep=*,leftmargin=\@listleftmargin,itemsep=0pt,topsep=0pt,partopsep=0pt,parsep=\@myparskip}
}

\def\denseversion{
  \setcounter{denseversion}{1}
  \newgeometry{left=3cm,right=3cm,top=3cm}
  \mybaselinestretch{1.1}
  \myparskip{0.8ex}
  \normalfont
  \def\possiblelinebreak{}
  \fancyfoot[C]{\itshape{--$\,\,$\thepage$\,\,$--}}}

\def\possiblelinebreak{\\}

\renewcommand{\emph}[1]{\def\reserved@a{it}\ifx\f@shape\reserved@a\ul{#1}\else\textit{#1}\fi}

\def\setcrefnames{}
\newcommand{\mytableofcontents}{
   \ifnum\value{denseversion}=0
     \tableofcontents
     \setcrefnames 
   \else
     \renewcommand{\baselinestretch}{1.1}
     \setlength{\parskip}{0ex}
     \normalfont
     \begingroup
     \def\addvspace##1{\vskip0.4em}
     \tableofcontents
     \setcrefnames 
     \endgroup
     \renewcommand{\baselinestretch}{\@mybaselinestretch}
     \setlength{\parskip}{\@myparskip}
     \normalfont
   \fi}

\newlength{\zeilenlaenge}
\def\putindent#1{
  \settowidth{\zeilenlaenge}{#1}
  \ifnum\zeilenlaenge>\textwidth
    #1
  \else
    \noindent #1
  \fi
}

\hypersetup{
    linktocpage = true,
    colorlinks,
    linkcolor=[RGB]{0,0,120},
    citecolor=[RGB]{0,0,120},
    urlcolor=[RGB]{0,0,120}
}

\def\pdfdaten{
  \hypersetup{
    pdftitle = {\@title},
    pdfauthor = {\@author},
    pdfkeywords = {\@keywords},    
    bookmarksopen = true,
    bookmarksopenlevel = 1
  }}  

\def\showkeywords{\begin{flushleft}\footnotesize\textbf{Keywords}: \@keywords\end{flushleft}}
\def\showmsc{\begin{flushleft}\footnotesize\textbf{MSC 2010}: \@msc\end{flushleft}}

\def\tocsection#1{\section*{#1}\addcontentsline{toc}{section}{#1}}

\def\mytitle{}
\def\zmptitle{
  \begin{tabular}{cc}
    \begin{minipage}[c]{0.4\textwidth}
      \begin{flushleft}
        \includegraphics[width=110pt]{../../tex/zmp}
      \end{flushleft}  
    \end{minipage}&
    \begin{minipage}[c]{0.55\textwidth}
      \begin{flushright}
      {\small\sf\@preprint}
      \end{flushright}
    \end{minipage}
  \end{tabular}
  \vskip 2cm}

\def\maketitle{
  \pdfdaten
  \noindent
  \mytitle
  \begin{center}
    \LARGE\@title\\
    \if!\@subtitle!\else\smallskip\LARGE\@subtitle\\\fi
    \bigskip
    \if!\@author!\else\bigskip\large\@author\\\fi
    \ifnum\value{denseversion}=0
      \if!\@adressstringa!\else\bigskip\normalsize\@adressstringa\\\fi
      \if!\@email!\else\ifnum\value{authorcounter}=1\bigskip\normalsize\textit{\@email}\\\else\fi\fi
    \else
    \fi
    \if!\@dedication!\else\bigskip\normalsize{\@dedication}\\\fi
  \end{center}
  \ifnum\value{denseversion}=0\vskip 1.5cm\else\vskip0.5cm\fi}

\def\kobib#1{
  \begin{raggedright}
  \ifnum\value{denseversion}=0\else\small\fi
  \Oldbibliography{#1/kobib}
  \bibliographystyle{#1/kobib}
  \end{raggedright}
  \ifnum\value{denseversion}=0\else
      \noindent
      \if!\@authorstringc!\else
        \ifnum\authortagsused=0\ifnum\value{authorcounter}>1\normalsize\@authorstringc\\[\medskipamount]\else\fi\else\normalsize\@authorstringc\\[\medskipamount]\fi
      \fi
      \if!\@adressstringb!\else\normalsize\@adressstringb\\{}\fi
      \ifnum\authortagsused=0
        \ifnum\value{authorcounter}=1
          \if!\@email!\else\linebreak\normalsize\textit{\@email}\\{}\fi
        \else
        \fi
      \else
      \fi
  \fi}

\let\Oldbibliography\bibliography
\def\bibliography#1{
  \begin{raggedright}
  \ifnum\value{denseversion}=0\else\small\fi
  \Oldbibliography{#1}
  \end{raggedright}
  \ifnum\value{denseversion}=0\else
      \medskip
      \noindent
      \if!\@authorstringc!\else
        \ifnum\authortagsused=0\ifnum\value{authorcounter}>1\normalsize\@authorstringc\\[\medskipamount]\else\fi\else\normalsize\@authorstringc\\[\medskipamount]\fi
      \fi
      \if!\@adressstringb!\else\normalsize\@adressstringb\\{}\fi
      \ifnum\authortagsused=0
        \ifnum\value{authorcounter}=1
          \if!\@email!\else\linebreak\normalsize\textit{\@email}\\{}\fi
        \else
        \fi
      \else
      \fi
  \fi
}

\newenvironment{commentfigure}{}
\newenvironment{sidewayscommentfigure}{\begin{minipage}}{\end{minipage}}
\newenvironment{displaycomment}{\begin{list}{}{\rightmargin=1cm\leftmargin=1cm}\item\sf\begin{small}}{\end{small}\end{list}}

\def\tocmark#1{}

\def\draftstamp#1{
  \def\tocmark##1{
    \ifnum\c@secnumdepth=0\section{##1}\fi
    \ifnum\c@secnumdepth=1\subsection{##1}\fi
    \ifnum\c@secnumdepth=2\subsubsection{##1}\fi
    \ifnum\c@secnumdepth=3\subsubsection{##1}\fi
  }
  \ifnum\value{comments}=0
    \gdef\@draft{DRAFT - Edited on \today\ by #1 - Comments are not displayed}
  \else
    \gdef\@draft{DRAFT - Edited on \today\ by #1 - Comments are displayed}
  \fi
  \fancyhead[C]{\footnotesize\tt\textcolor{red}{\@draft}}}

\def\skript{
  \renewenvironment{displaycomment}{}{}
  \ifnum\value{comments}=0
    \renewenvironment{example*}{\comment}{\endcomment}
    \renewenvironment{remark*}{\comment}{\endcomment}
  \else\fi
  \parindent=0mm	
}

\newcounter{sectioning}
\def\tsection#1{\ifnum\value{sectioning}=0\section{#1}\fi}
\def\lsection#1{
  \ifnum\value{sectioning}=1
    \clearpage
    \def\thesection{\lektionname~\arabic{section}:}
    \section{#1}
    \def\thesection{\arabic{section}}
  \fi}
\def\tsubsection#1{\ifnum\value{sectioning}=0\subsection{#1}\fi}
\def\lsubsection#1{\ifnum\value{sectioning}=1\subsection{#1}\fi}
\def\tsubsubsection#1{\ifnum\value{sectioning}=0\subsubsection{#1}\fi}
\def\lsubsubsection#1{\ifnum\value{sectioning}=1\subsubsection{#1}\fi}

\mybaselinestretch{}
\setenumstandard

\makeatother

%% file: makros.tex
\usepackage{amssymb}
\usepackage{amsmath}
\usepackage{amstext}
\usepackage{mathrsfs}
\usepackage{euscript}
\usepackage[all,cmtip]{xy}
\usepackage{xstring}
\usepackage{slashed}
\usepackage{tensor}
\usepackage{tikz}

\entrymodifiers={+!!<0pt,\fontdimen22\textfont2>}

\def\Z {\mathbb{Z}}

\def\R {\mathbb{R}}

\def\T {\mathbb{T}}
\def\im{\mathrm{i}}
\def\id{\mathrm{id}}

\def\hc#1{\mathrm{h}_{#1}}
\def\h {\mathrm{H}}

\def\subset{\subseteq}

\def\sep{\;|\;}

\renewcommand{\varepsilon}{\epsilon}

\makeatletter
\renewenvironment{proof}[1][\nameProof]
  {\par\pushQED{\qed}%
   \normalfont \topsep6\p@\@plus6\p@\relax
   \trivlist
   \item[\hskip\labelsep
         \itshape
         #1\@addpunct{.}]
  \leavevmode}
  {\popQED\endtrivlist\@endpefalse}
\makeatother

\def\notebox#1#2{\begin{minipage}[b]{#1}\sloppy\renewcommand{\baselinestretch}{0.8}\footnotesize \begin{center}#2\end{center}\end{minipage}}

\def\mquad{\hspace{-2em}}

\newcommand{\arr}[1][r]{\ar@<0.7ex>[#1]\ar@<-0.7ex>[#1]}
\newcommand{\arrr}[1][r]{\ar@<1.4ex>[#1]\ar[#1]\ar@<-1.4ex>[#1]}
\newcommand{\arrrr}[1][r]{\ar@<2.1ex>[#1]\ar@<-2.1ex>[#1]\ar@<0.7ex>[#1]\ar@<-0.7ex>[#1]}

\def\stackref#1#2{\stackrel{\text{\ref{#1}}}{#2}}
\def\eqref#1{\stackref{#1}{=}}

\newlength{\myeqt} 
\newlength{\myeqs} 
\newlength{\myeqm} 
\newlength{\myeqn} 
\newcommand\eqtext[2][\myeqn]{\symtext[#1]{#2}{=}}
\newcommand\symtext[3][\myeqn]{
  \settowidth{\myeqt}{#2}
  \settowidth{\myeqs}{$#3$}
  \addtolength{\myeqs}{\the\myeqm}
  \ifdim\myeqt>\myeqs
    \stackrel{\hspace{-#1}\notebox{#1}{\medskip #2 \\ $\downarrow$\smallskip}\hspace{-#1}}{#3}
  \else
    \stackrel{\text{#2}}{#3}
  \fi}
\newcommand\eqcref[2][\myeqn]{\symcref[#1]{#2}{=}}

\newcommand\symcref[3][\myeqn]{\symtext[#1]{\cref{#2}}{#3}}

\def\brackets#1{\IfStrEq{#1}{-}{}{(#1)}}
\def\subindex#1{\IfStrEq{#1}{-}{}{_{#1}}}

\newcommand{\alxydim}[2]{\begin{aligned}\xymatrix#1{#2}\end{aligned}}

\makeatletter
\def\bigset#1#2{\left\lbrace\;\begin{minipage}[c]{#1}\begin{center}#2\end{center}\end{minipage}\;\right\rbrace}

\makeatother

\newlength{\myl}
\newcommand\sheaf[1]{\unitlength 0.1mm
  \settowidth{\myl}{$#1$}
  \addtolength{\myl}{-0.8mm}
  \begin{picture}(0,0)(0,0)
  \put(2,0){\text{\underline{\hspace{\myl}}}}
  \end{picture}#1\hspace{-0.15mm}}

\def\ddt#1#2#3{\left.\frac{\mathrm{d}^{\IfStrEq{#1}{1}{}{#1}}}{\mathrm{d}#2}\IfStrEq{#2}{#3}{\right.}{\right|_{#3}}}

\def\AUT{\mathrm{AUT}}

\newcommand{\ueins}{{\mathrm{U}}(1)}

\newcommand{\og}[1]{{\mathrm{O}}\brackets{#1}}

\def\fun{\mathcal{F}un}

\def\ev{\mathrm{ev}}

\def\pr{{\mathrm{pr}}}
\def\B{\mathcal{B}\hspace{-0.03em}}

\newlength{\widthtmp}
\def\length#1{\settowidth{\widthtmp}{#1}\the\widthtmp}



%% file: hyphenation_en.tex
\usepackage[latin1]{inputenc}
\usepackage[english]{babel}
\usepackage[english]{cleveref}

\def\refname{References}

\def\quot#1{``#1''}

\def\quand{\quad\text{ and }\quad}
\def\quomma{\quad\text{, }\quad}

\def\quith{\quad\text{ with }\quad}

\def\nameTheorem{Theorem}
\def\nameTheorems{Theorems}
\def\nameDefinition{Definition}
\def\nameDefinitions{Definitions}
\def\nameCorollary{Corollary}
\def\nameCorollaries{Corollaries}
\def\nameProposition{Proposition}
\def\namePropositions{Propositions}
\def\nameConstruction{Construction}
\def\nameConstructions{Constructions}
\def\nameLemma{Lemma}
\def\nameLemmas{Lemmas}
\def\nameRemark{Remark}
\def\nameRemarks{Remarks}
\def\nameProblem{Problem}
\def\nameProblems{Problems}
\def\nameExample{Example}
\def\nameExamples{Examples}
\def\nameExercise{Exercise}
\def\nameExercises{Exercises}

\def\nameProof{Proof}
\def\nameEquation{Eq.}
\def\nameEquations{Eqs.}
\def\nameSection{Section}
\def\nameSections{Sections}
\def\nameAppendix{Appendix}
\def\nameAppendixs{Appendices}
\def\nameFigure{Figure}
\def\nameFigures{Figures}